\documentclass[article]{amsart}

\usepackage{amssymb}
\usepackage{url}
\usepackage{color}
\newtheorem{theorem}{Theorem}[section]
\newtheorem*{theorem*}{Theorem}
\newtheorem*{lemma*}{Lemma}

\newtheorem{construction}[theorem]{Construction}
\newtheorem{conjecture}[theorem]{Conjecture}
\newtheorem*{conjecture*}{Conjecture}
\newtheorem{corollary}[theorem]{Corollary}

\newtheorem{definition}[theorem]{Definition}
\newtheorem{example}[theorem]{Example}

\newtheorem{lemma}[theorem]{Lemma}
\newtheorem{notation}[theorem]{Notation}

\newtheorem{proposition}[theorem]{Proposition}
\newtheorem{remark}[theorem]{Remark}
\newtheorem*{remark*}{Remark}

\DeclareMathOperator{\Cov}{Cov}
\begin{document}
\title{W-Operator and Differential Equation for 3-Hurwitz Number}

\author{Hao Sun}


\maketitle

\begin{abstract}
We consider a new type of Hurwitz number, the number of ordered transitive factorizations of an arbitrary permutation into $d$-cycles. In this paper, we focus on the special case $d=3$. The minimal number of transitive factorizations of any permutation into $3$-cycles has been worked out by David, Goulden and Jackson. Also, such factorizations for transpositions, the case $d=2$, have been considered by Crescimanno and Taylor. Goulden and Jackson have proved the differential equation for the generating series of simple Hurwitz numbers. Based on their results, we use $W$-operator to prove a differential equation for the generating function of the new type Hurwitz number.
\end{abstract}

\section{Introduction}

The Hurwitz enumeration problem \cite{Lando} aims at classifying all $n$-fold coverings of $S^2$ (or $\mathbb{C}P^1$) with $k$ branch points $\{z_1,...,z_k\}$. Given such a covering, each branch point $z_i$ corresponds to a unique permutation $\sigma_i$ of type $\lambda_i$ in $S_n$, where $\lambda_i$ is a partition of $n$. The number of such connected $n$-coverings is finite and denoted by $\Cov_n(\lambda_1,...,\lambda_k)$. Alternatively, $\Cov_n(\lambda_1,...,\lambda_k)$ is the number of $k$-tuples $(\sigma_1,...,\sigma_k) \in S^{k}_n$ satisfying the following conditions \cite{Carrel} \cite{Lando},

\begin{enumerate}
	\item[(1)]  $\sigma_i$ is of type $\lambda_i$,
	\item[(2)]  $\sigma_1...\sigma_k=1$,
	\item[(3)]  The group generated by $\sigma_1,...,\sigma_k$ is transitive on the set $\{1,...,n\}$.
\end{enumerate}

Now, we consider a special case of the Hurwitz enumeration problem. Given $\alpha$ a partition of $n$, we define the simple Hurwitz number as
\begin{align*}
	h_k(\alpha)=\Cov_n(1^{n-2}2,...,1^{n-2}2,\alpha).
\end{align*}
It is the number of $(k+1)$-tuples $(\tau_1,...,\tau_k,\sigma) \in S^{k+1}_n$ satisfying the following conditions
\begin{enumerate}
	\item[(1)]  $\tau_i$ are transpositions (or of type $1^{n-2}2$), where $1 \leq i \leq k$, and $\sigma$ is of type $\alpha$,
	\item[(2)]  $\tau_1...\tau_k \sigma=1$,
	\item[(3)]  The group generated by $\{\tau_1,...,\tau_k\}$ is transitive on the set $\{1,...,n\}$.
\end{enumerate}
I.P.Goulden and D.M.Jackson \cite{MR1396978} slightly change the problem by adding an extra condition
\begin{enumerate}
	\item[(4)]  Given any permutation $\sigma$ in $S_n$, the number of transpositions $k$ is minimal with respect to conditions $(1)$,$(2)$ and $(3)$.
\end{enumerate}
Under this condition, we write $h(\alpha)$ for $h_k(\alpha)$ and call it the minimal simple Hurwitz number.

Given any element $\sigma \in S_n$, denote by $\mu(\sigma)$ \cite{Cre} \cite{MR1396978} the minimal number $k$ of transpositions satisfying conditions $(1)$,$(2)$ and $(3)$ as in condition $(4)$. Before introducing the generating function, we first introduce some notations. We write $\alpha \vdash n$ if $\alpha$ is a partition of $n$, i.e. if $\alpha=(\alpha_1,...,\alpha_l)$, then $\alpha_1+...+\alpha_l=n$ and $\alpha_1 \geq ... \geq \alpha_l$. Let $p_\alpha=p_{\alpha_1}...p_{\alpha_l}$, where $p_i$ are variables. We will give another definition of $p_i$ in Section 3. If $\sigma$ and $\sigma'$ are of the same type, then $\mu(\sigma)=\mu(\sigma')$, which will be proved in Lemma 2.1. Hence, sometimes we use the notation $\mu(\alpha)$ in this paper for $\mu(\sigma)$, where $\alpha$ is the partition corresponding to $\sigma$. Now we come to the generating function of minimal simple Hurwitz numbers
\begin{align*}
F_2(z,p)=F_2(z,p_1,p_2,...)=\sum_{n \geq 1}\sum_{\alpha \vdash n}h(\alpha)\frac{z^n}{n!}\frac{1}{\mu(\alpha)!}p_\alpha.
\end{align*}
The cut-and-join operator $\Delta$ was introduced by Goulden \cite{MR1249468}. It is an infinite sum of differential operators in variables $p_1,p_2,... \text{ .}$ Goulden and Jackson \cite{MR1396978} use this operator to calculate the minimal simple Hurwitz number.

Mironov, Morosov and Natanzon construct $W$-operators $W([d])$ \cite{MR2864467}, where $d$ is a positive integer. They are differential operators acting on the space $\mathbb{C}[[X_{ij}]]_{i,j \geq 1}$ of formal series in variables $X_{ij}$ $(i,j \geq 0)$, where $X_{ij}$ are coordinate functions on the infinite matrix. A subring of $\mathbb{C}[[X_{ij}]]_{i,j \geq 1}$ is $\mathbb{C}[p_1,p_2,...]$, where $p_k=Tr(X^k)$ and $X=(X_{ij})_{i,j \geq 1}$. A direct calculation shows that $W([2])$ is the cut-and-join operator $\Delta$ on the ring $\mathbb{C}[p_1,p_2,...]$. Mironov et al. apply the $W$-operator to the Hurwitz enumeration problem \cite{MR2864467} and find a new equation about the generating function of some special Hurwitz numbers.

In Section 2, we review some results about the minimal simple Hurwitz number and the main steps to calculate it.

In Section 3, we review the definition of the $W$-operator and some properties.

In Section 4, we define a new type of Hurwitz number, the $d$-Hurwitz number,
\begin{align*}
h_k^{d}(\alpha)=\Cov_n(1^{n-d}d,...,1^{n-d}d,\alpha).
\end{align*}
It is the number of $(k+1)$-tuples $(\delta_1,...,\delta_k,\sigma) \in S^{k}_n$ satisfying the following conditions
\begin{enumerate}
	\item[(1)]  $\delta_i$ are $d$-cycles (or of type $1^{n-d}d$), where $1 \leq i \leq k$, and $\sigma$ is of type $\alpha$,
	\item[(2)]  $\delta_1...\delta_k \sigma=1$,
	\item[(3)]  The group generated by $\{\delta_1,...,\delta_k\}$ is transitive on the set $\{1,...,n\}$.
\end{enumerate}
Similarly, we define the minimal $d$-Hurwitz number $h^{d}(\alpha)$ by adding another condition
\begin{enumerate}
\item[(4)] $k$ is minimal with respect to the conditions $(1)$,$(2)$ and $(3)$.
\end{enumerate}
Denote by $\mu^d(\sigma)$ the minimal number. Compared to the simple Hurwitz number, we replace every transposition by a $d$-cycle. The simple Hurwitz number is therefore the $2$-Hurwitz number $h_k(\alpha)=h^2_k (\alpha)$ and the minimal simple Hurwitz number is the minimal $2$-simple Hurwitz number $h(\alpha)=h^2(\alpha)$.

Given any permutation $\sigma$, we review the calculation of the number $\mu^d(\sigma)$ when $d=3$ \cite{MR1797682}.
\begin{lemma}[\textbf{\ref{408}}]
Let $n$ be a positive integer and let $\sigma$ be a permutation in the alternating group $A_n$, i.e. $\sigma$ is the product of $3$-cycles. We decompose $\sigma=\rho_1...\rho_l$ in disjoint cycles $\rho_i$, then we have
\begin{enumerate}
	\item[(1)]  if $n$ is odd, $\mu^3(\sigma)=\frac{n-1}{2}+\frac{l-1}{2}$,
	\item[(2)]  if $n$ is even, $\mu^3(\sigma)=\frac{n}{2}+\frac{l-2}{2}$.
\end{enumerate}
In conclusion, we have $\mu^3(\sigma)=\frac{n+l-2}{2}$.
\end{lemma}

In Section 5, we construct the generating function for minimal $d$-Hurwitz numbers,
\begin{align*}
F_d(z,p_1,p_2,...)=\sum_{n \geq 1}\sum_{\alpha \vdash n}h^{d}(\alpha)\frac{z^n}{n!}\frac{1}{\mu^d(\alpha)!}p_\alpha.
\end{align*}
$\widetilde{W}([3])$ is a "differential operator" in variables $p_i$, $i \geq 1$, defined in Definition \ref{504}. We prove the following theorem.
\begin{theorem}[\textbf{\ref{505}}]
The generating function $F_3(z,p)$ satisfies the following relation
\begin{align*}
& \widetilde{W}([3])F_3-\sum_{i,j,k \geq 1}(i+j+k)p_{i+j+k} \frac{\partial F_3}{\partial p_{i+j+k}}\\
=& \frac{1}{2}(z \frac{\partial F_3}{\partial z}+\sum_{i \geq 1} p_i \frac{\partial F_3}{\partial p_i}-2F_3).
\end{align*}
\end{theorem}

In Section 6, we have a conjecture about the differential equations satisfied by $F_d$ and $\widetilde{F}_d$, where
\begin{align*}
\widetilde{F}_d=\sum_{n \geq 1}\sum_{\alpha\vdash n}h^d(\alpha)\frac{z^n}{n!}\frac{u^{\mu^d(\alpha)}}{\mu^d(\alpha)!}p_\alpha.
\end{align*}
Before we state the conjectures, we review some properties of $W([d])$ and define a new operator $\widetilde{HW}([d])$. $W([d])$ can be written as the sum of $d!$ summations, each of which corresponds uniquely to a permutation in $S_d$ \cite{Sun1}, i.e.
\begin{align*}
W([d])=\sum_{\beta \in S_d} FS_\beta,
\end{align*}
where $FS_\beta$ is the summation corresponding to $\beta \in S_d$. We define the degree for the summation $FS_\beta$ (see Definition \ref{601}). The degree of the summation $FS_\beta$, $\beta \in S_n$ is at most $d+1$ (see Section 6 and \cite{Sun4}). Finally, we define another operator $\widetilde{HW}([d])$, which is sum of all summations with degree $d+1$ (see Definition \ref{602}). The conjecture is as following.
\begin{conjecture}[\textbf{\ref{603}}]
Given any positive integer $d$, we have
\begin{align*}
\frac{\partial \widetilde{F}_d}{\partial u}=\widetilde{HW}([d])\widetilde{F}_d.
\end{align*}
\end{conjecture}
If the above conjecture is true, we will get the following corollary by taking $u=1$.
\begin{corollary}[\textbf{\ref{605}}]
\begin{align*}
\widetilde{HW}([d])F_d=\frac{1}{d-1}\left( z\frac{\partial F_d}{\partial z}+ \sum_{i \geq 1}p_i \frac{\partial F_d}{\partial p_i} -2F_d \right).
\end{align*}
\end{corollary}
In this paper, we prove the special case $d=3$ (see Theorem \ref{505}). For the general case, we believe similar to the case $d=3$ can be given, but will be very complicated to write down.
\section{Simple Hurwitz Number}

In this section, we review some results about the minimal simple Hurwitz number.

First, we want to make a short remark about the disjoint cycles in a permutation $\sigma$. Every permutation $\sigma$ in $S_n$ can be written uniquely as a product of disjoint cycles, where the cycles are $d$-cycles, $d \geq 2$. But, in this paper, the product of disjoint cycles also includes the fixed points or "1"-cycles. For example, let's consider $\sigma=(1 2 3) \in S_4$. In this paper, if we write $\sigma$ as the product of disjoint cycles, it is $(1 2 3)(4)$.

Now, we come to the minimal simple Hurwitz number.

Recall for given any permutation $\sigma$, $\mu(\sigma)$ is the minimal number of transpositions satisfying condition (1),(2),(3),(4) in the beginning of the introduction. If two permutations $\sigma$ and $\sigma'$ are of the same type, then $\mu(\sigma)=\mu(\sigma')$. This statement is a direct consequence of the following lemma.

\begin{lemma}\label{201}
Given any element $\sigma \in S_n$, we write $\sigma=\rho_1...\rho_{l(\sigma)}$ as the product of disjoint cycles, where $l(\sigma)$ is the number of disjoint cycles of $\sigma$, then we have
\begin{align*}
\mu(\sigma)=n+l(\sigma)-2.
\end{align*}
\end{lemma}

\begin{proof}
\cite{MR1396978}, Proposition $2.1$.
\end{proof}
Since $\mu(\sigma)=\mu(\sigma')$, when $\sigma$ and $\sigma'$ are of the same type $\alpha$, sometimes we use the notation $\mu(\alpha)$ in this paper for $\mu(\sigma)$.
\begin{definition}
Assume $G$ is a group, $X$ is a set and there is a group homomorphism $G \rightarrow Aut(X)$. A subset $X'$ of $X$ is a connected component of $X$ (w.r.t $G$) if $G$ acts transitive on $X'$.
\end{definition}

Now consider a permutation $\sigma \in S_n$, we can write $\sigma$ into the product of disjoint cycles uniquely up to reordering,
\begin{align*}
\sigma=\rho_1...\rho_l.
\end{align*}
In his case, Lemma \ref{201} tells us $\mu(\sigma)=n+l-2$. Then, we can find $\mu(\sigma)$ transpositions $\tau_1,...,\tau_{\mu(\sigma)}$ such that
\begin{align*}
\sigma=\tau_1\tau_2...\tau_{\mu(\sigma)},
\end{align*}
and the group generated by $\{\tau_1,...,\tau_{\mu(\sigma)}\}$ acts transitively on $\{1,...,n\}$. We are interested in the product
\begin{align*}
\sigma'=\tau_1 \sigma=\tau_2...\tau_{\mu_\sigma}.
\end{align*}
If $\tau_1=(j_2 \text{ } j_1)$, $j_1 \neq j_2 \in \{1,...,n\}$, then there are two cases for $\sigma'=\tau_1 \sigma$.

\begin{lemma}\label{202}
\begin{enumerate}
    \item If $j_1,j_2$ occur in a single disjoint cycle (say $\rho_1$), then
    \begin{align*}
    \tau_1 \rho_1=(j_1...)(j_2...)
    \end{align*}
    is a product of two disjoint cycles. In this case, the set $\{1,...,n\}$ has two connected component under the action of the group generated by $\{\tau_i, 2 \leq i \leq \mu(\sigma)\}$. One contains $j_1$ and the other one contains $j_2$. In this case, $\tau_1$ is known as a cut-operator relative to $\sigma$.
	\item If $j_1,j_2$ occur in two different disjoint cycles (say $j_1 \in \rho_1$ and $j_2 \in \rho_2$), then
    \begin{align*}
    \tau_1 \rho_1 \rho_2=(j_1...j_2...)
    \end{align*}
    is the product of a single cycle and the action of the group generated by $\{\tau_i, 2 \leq i \leq \mu(\sigma)\}$ is transitive on the set $\{1,...,n\}$. In this case, $\tau_1$ is known as a join-operator relative to $\sigma$.
\end{enumerate}
\end{lemma}

\begin{proof}
We consider the first case. The factorization of $\tau_1 \rho_1$ in two disjoint cycles is a simple calculation of disjoint cycles. Since $\sigma$ has $l$ disjoint cycles, so $\sigma'$ has $l+1$ disjoint cycles in this case. We know $\mu(\sigma')=\mu(\sigma)+1$ by Lemma \ref{201}. Now, $\sigma'$ is a product of $\mu(\sigma)-1$ transpositions $\tau_2,...,\tau_{\mu(\sigma)}$. Hence, the group generated by $\{\tau_2,...,\tau_{\mu(\sigma)}\}$ is not transitive on the set $\{1,...,n\}$. We know the group generated by $\{\tau_1,\tau_2,...,\tau_{\mu(\sigma)}\}$ is transitive on the set $\{1,...,n\}$. So, $\{1,...,n\}$ has two connected component with respect to the action of $\{\tau_2,...,\tau_{\mu(\sigma)}\}$ and one contains $j_1$, the other contains $j_2$.

The second statement is easy to verify by a similar argument.
\end{proof}

Goulden et al. use the cut-and-join operator to prove that the generating function of minimal simple Hurwitz numbers satisfies a special differential equation. Recall the generating function of minimal simple Hurwitz numbers is
\begin{align*}
F_2(z,p_1,p_2,...)=\sum_{n \geq 1}\sum_{\alpha \vdash n}h(\alpha)\frac{z^n}{n!}\frac{1}{\mu(\alpha)!}p_\alpha.
\end{align*}

\begin{lemma}\label{203}
The generating function $F_2(z,p_1,p_2,...)$ satisfies the following equation
\begin{align*}
\frac{1}{2}\sum_{i,j \geq 1}(p_{i+j}i \frac{\partial F_2}{\partial p_i} j \frac{\partial F_2}{\partial p_j}+(i+j)p_i p_j \frac{\partial F_2}{\partial p_{i+j}})-z \frac{\partial F_2}{\partial z} - \sum_{i \geq 1}p_i \frac{\partial F_2}{\partial p_i}+2F_2=0.
\end{align*}
\end{lemma}

\begin{proof}
\cite{MR1396978}, Lemma $2.2$.
\end{proof}
Goulden et al. use this equation to calculate the minimal simple Hurwitz number.
\begin{theorem}\label{204}
Let $\alpha=(\alpha_1,...,\alpha_l)\vdash n$, for $n,l \geq 1$. Then
\begin{align*}
h(\alpha)=n^{l-3}(n+l-2)!\prod^{l}_{j=1} \frac{(\alpha_j)^{\alpha_j} }{(\alpha_j-1)!}.
\end{align*}
\end{theorem}

\begin{proof}
\cite{MR1396978}, Theorem $1.1$.
\end{proof}

\section{W-Operator}

\begin{definition}\label{301}
	A variable matrix $X$ is an infinite matrix with variable $X_{ab}$ in the $(a,b)$-entry. Generally, $X:=(X_{ab})_{a,b\geq1}$ and all $X_{ab}$ are assumed to commute with each other.
\end{definition}

\begin{definition}\label{302}
	Define $p_{k}$ the trace of $X^{k}$, i.e., $p_{k}=tr(X^{k})$. $p_{k}$ is a power series in $\mathbb{C}[[X_{ab}]]_{a,b \geq 1}$. $\mathbb{C} [p_{1},p_{2},...]$ is a polynomial ring with infinitely many variables $p_{k}$.
\end{definition}

\begin{definition}\label{303}
	The operator matrix $D$ is an infinite matrix with $D_{ab}$ in the $(a,b)$-entry, where $D_{ab}=\sum\limits_{c=1}^{\infty} X_{ac} \frac{\partial}{\partial X_{bc}}$.
\end{definition}

\begin{definition}\label{304}
	The normal ordered product of $D_{ab}$ and $D_{cd}$ is
	\begin{align*}
		:D_{ab}D_{cd}:=\sum_{e_1,e_2 \geq 1}X_{ae_{1}}X_{ce_{2}} \frac{\partial}{\partial X_{be_{1}}} \frac{\partial}{\partial X_{de_{2}}}.	
	\end{align*}
	Similarly, the normal product $:\prod\limits_{i=1}^{d}D_{a_i b_i}:$ is
	\begin{align*}
		:\prod\limits_{i=1}^{d}D_{a_i b_i}:	= \sum_{e_1,...,e_d \geq 1} (\prod\limits_{i=1}^{d}X_{a_i e_i}\prod\limits_{i=1}^{d}\frac{\partial}{\partial X_{b_i e_i}}).
	\end{align*}
\end{definition}

\begin{definition}\label{305}
	For any positive integer d, we define the W-operator $W([d])$ as
	\begin{align*}
		W([d]):=\frac{1}{d}:tr(D^{d}):.
	\end{align*}
\end{definition}

\begin{theorem}\label{306}
$W([d])$ is a well-defined operator on $\mathbb{C}[p_{1},p_{2},...]$ and it can be written as the sum of $d!$ summations, each of which corresponds to a unique permutation in $S_d$.	
\end{theorem}

\begin{proof}
	It is proved in \cite{Sun1}, Theorem 3.15.
\end{proof}

\begin{example}\label{3006}
The first example is $W([2])$, which is also known as the cut-and-join operator,
\begin{align*}
W([2])= \frac{1}{2}\sum_{i \geq 1}\sum_{j\geq 1}(ijp_{i+j}\frac{\partial^2}{\partial p_i \partial p_j}+(i+j)p_i p_j \frac{\partial}{\partial p_{i+j}}).
\end{align*}
The second example is $W([3])$,
\begin{align*}
		W([3]) = \frac{1}{3}\sum_{i_1,i_2,i_3 \geq 1}
		(&i_1i_2i_3 p_{i_1+i_2+i_3}\frac{\partial^3}{\partial p_{i_1} \partial p_{i_2}\partial p_{i_3}}+  \\
		+&i_1(i_2+i_3)p_{i_1+i_3}p_{i_2}\frac{\partial^2}{\partial p_{i_1} \partial p_{i_2+i_3}}+ \\
		+&i_2(i_1+i_3)p_{i_1+i_2}p_{i_3}\frac{\partial^2}{\partial p_{i_2} \partial p_{i_1+i_3}}+ \\
		+&i_3(i_1+i_2)p_{i_3+i_2}p_{i_1}\frac{\partial^2}{\partial p_{i_3} \partial p_{i_1+i_2}}+ \\
		+&(i_1+i_2+i_3)p_{i_1}p_{i_2}p_{i_3} \frac{\partial}{\partial p_{i_1+i_2+i_3}}+ \\
		+&(i_1+i_2+i_3)p_{i_1+i_2+i_3}\frac{\partial}{\partial p_{i_1+i_2+i_3}}) .
\end{align*}
\end{example}

\begin{definition}\label{307}
	Define a map
	\begin{align*}
	\Phi: \mathbb{C}S_n \rightarrow \mathbb{C}[p_1,p_2,...]
	\end{align*}
	such that for $\sigma \in S_n$, we have
	\begin{align*}
		\Phi(\sigma)=p_{\alpha}=p_{\alpha_1}...p_{\alpha_l},
	\end{align*}
	where $\alpha=(\alpha_1,...,\alpha_m)$ is the partition corresponding to $\sigma$.
\end{definition}
Note $\Phi(\sigma)=\Phi(\sigma')$ if $\sigma$ and $\sigma'$ have the same type $\alpha$, and sometimes we also write
\begin{align*}
\Phi(\alpha)=p_{\alpha}=p_{\alpha_1}...p_{\alpha_l}.
\end{align*}

\begin{definition}\label{308}
	If $\alpha$ is a partition of a positive integer $n$, define the element $K_\alpha \in \mathbb{C}S_n$ as
	\begin{align*}
		K_\alpha = \sum_{\sigma \in S_n, \atop \sigma \text{ is of type } \alpha}\sigma.
	\end{align*}
	Clearly, $K_\alpha$ is in the center of in $\mathbb{C}S_n$.
\end{definition}

For example, $\Phi(K_\alpha)=c_\alpha p{_\alpha}$, where $c_\alpha$ is the number of all $\sigma \in S_n$ of type $\alpha$.

\begin{notation}\label{309}
	Given a partition $\lambda=(\lambda_1,...,\lambda_m)$ of a positive integer $n$, we can write it as
	\begin{align*}
		\lambda=1^{k_1}2^{k_2}...s^{k_s}
	\end{align*}	
	where $k_i$ is the number of times the integer $i$ appears in the partition $\lambda$. For example, if $\lambda=(1^{n-d}d)$, then $K_{1^{n-d}d}$ is a central element in $\mathbb{C}S_n$, which is the sum of all $d$-cycles in $S_n$.
\end{notation}

\begin{proposition}
For any $g \in \mathbb{C}S_n$,
	\begin{align*}
		\Phi(K_{1^{n-2}2}g)=W([2]) \Phi(g).
	\end{align*}
\end{proposition}

\begin{proof}
\cite{MR1249468}, Proposition 3.1.
\end{proof}

We generalize this property to $d$-cycles and $W([d])$.
\begin{theorem}\label{310}
	For any $g \in \mathbb{C}S_n$ and any positive integer $d$ such that $2 \leq d \leq n$, we have
	\begin{align*}
		\Phi(K_{1^{n-d}d}g)=W([d]) \Phi(g).
	\end{align*}
\end{theorem}

\begin{proof}
It is proved in \cite{Sun2}, Theorem 2.12.
\end{proof}
\section{Generalized Hurwitz Number}

The simple Hurwitz number $h_k(\alpha)=\Cov_n(1^{n-2}2,...,1^{n-2}2,\alpha)$ is the number of all $n$-fold coverings with $k+1$ branch points, where $k$ of them correspond to transpositions and the last branch point corresponds to a permutation of type $\alpha$. Now, we define a new type of Hurwitz number by replacing all transpositions by $d$-cycles.

\begin{definition}\label{401}
Given positive integers $n,k,d$, we define the $d$-Hurwitz number
\begin{align*}
h_k^{d}(\alpha)=\Cov_n(1^{n-d}d,...,1^{n-d}d,\alpha).
\end{align*}
It is the number of $(k+1)$-tuples $(\delta_1,...,\delta_k,\sigma) \in S^{k}_n$ satisfying the following conditions
\begin{enumerate}
	\item[(1)]  $\delta_i$ are $d$-cycles (or of type $(1^{n-d}d)$), where $1 \leq i \leq k$, and $\sigma$ is of type $\alpha$,
	\item[(2)]  $\delta_1...\delta_k=\sigma$,
	\item[(3)]  The group generated by $\{\delta_1,...,\delta_k\}$ is transitive on the set $\{1,...,n\}$.
\end{enumerate}
Furthermore, we define the minimal $d$-Hurwitz number $h^{d}(\alpha)=h^d_k(\alpha)$, where $k$ satisfies another condition
\begin{enumerate}
\item[(4)] $k$ is minimal with respect to the conditions $(1)$,$(2)$ and $(3)$.
\end{enumerate}
\end{definition}

\begin{remark}\label{4002}
Given any permutation $\sigma \in S_n$ of type $\alpha$, if $\sigma$ cannot be written as the product of $d$-cycles, then $h^d(\alpha)=0$, since Condition (2) in Definition \ref{401} is never satisfied. For example, consider the permutation $\sigma=(1 2)(3) \in S_3$ of type $\alpha=(21)$. We have $h^3(\alpha)=0$, since $(12)$ is an odd permutations and all $3$-cycles are even, so there is no $k$-tuple of $3$-cycles $(\delta_1,...,\delta_k)$ satisfying the second condition in Definition \ref{401}.
\end{remark}

\begin{definition}\label{402}
Given positive integers $n,d$ and a permutation $\sigma \in S_n$ which is of type $\alpha$, define $\mu^d(\sigma)$ to be the minimal number $k$ with respect to conditions $(1)$,$(2)$ and $(3)$ as in condition $(4)$.
\end{definition}
If $\sigma$ and $\sigma'$ are of the same type $\alpha$, then $\mu^d(\sigma)=\mu^d(\sigma')$. This statement can be easily proved by conjugation. So, sometimes we write the minimal number $\mu^d(\sigma)$ as $\mu^d(\alpha)$. The key lemma 4.10 in this section is first calculated by Goulden et al. in \cite{MR1797682} \cite{MR1249468}. Goulden et al. first calculate $\mu^2(\alpha)$ and use the property that any $3$-cycle (or $d$-cycle) can be decomposed as the product of two transpositions (or $d-1$ transpositions) to calculate $\mu^3(\sigma)$ (or $\mu^d(\sigma)$). We use a different method to calculate it. Construction \ref{403} is the most important construction in this section. Also, we only calculate $\mu^3(\sigma)$ in this section.

First, we will focus on how to calculate the product of a $3$-cycle $\omega$ and an arbitrary permutation $\sigma=\rho_1...\rho_l \in S_n$, where $\sigma=\rho_1...\rho_l$ is the decomposition of $\sigma$ into the product of disjoint cycles.
\begin{construction}\label{403}
Assume $\omega=(j_3 \text{ } j_2 \text{ } j_1)$, where $j_1,j_2,j_3$ are distinct integers in $\{1,...,n\}$. We are going to calculate $\omega \sigma$ according to the occurrences of $j_1,j_2,j_3$ in the disjoint cycles appearing in $\sigma$. If we consider the tuple $[j_3,j_2,j_1]$, there are $6$ cases with respect to the tuple $[j_3,j_2,j_1]$ and each one corresponds to a specific element of $S_3$,
\begin{enumerate}
		\item $\sigma=(j_1...)(j_2...)(j_3...)...,$
		\item $\sigma=(j_1...)(j_2...j_3...)...$,
		\item $\sigma=(j_1...j_3...)(j_2...)...$,
		\item $\sigma=(j_1...j_2...)(j_3...)...$,
		\item $\sigma=(j_1...j_2...j_3...)...$,
		\item $\sigma=(j_1...j_3...j_2...)...$.
\end{enumerate}
This has been discussed in \cite{Sun2}, Construction 4.11. Now we go back to the $3$-cycle $\omega$. Clearly, we have
\begin{align*}
\omega=(j_2 \text{ } j_1 \text{ } j_3)=(j_1 \text{ } j_3 \text{ } j_2).
\end{align*}
It means if we write $\omega=(j_3 \text{ } j_2 \text{ } j_1)$, then it corresponds to Case (2), if $\omega=(j_1 \text{ } j_3 \text{ } j_2)$, then it is Case (3) and if $\omega=(j_2 \text{ } j_1 \text{ } j_3)$, then it is Case (4). So, we can combine Case (2),(3),(4). In conclusion, we have $4$ cases with respect to the given $3$-cycle $\omega=(j_3 \text{ } j_2 \text{ } j_1)$,
\begin{enumerate}
		\item $\sigma=(j_1...)(j_2...)(j_3...)...$,
		\item $\sigma=(j_1...j_2...)(j_3...)...$,
		\item $\sigma=(j_1...j_2...j_3...)...$,
		\item $\sigma=(j_1...j_3...j_2...)...$.
\end{enumerate}
Clearly, for any element $\sigma \in S_n$, it falls into one and only one case with respect to $\omega$.

Let's consider Case (2) $\sigma=(j_1\colorbox{red}{...}j_2\colorbox{blue}{...})(j_3\colorbox{green}{...})...$, where the red dots represent the digits after $j_1$ before $j_2$, the blue dots represent the other digits after $j_2$ before $j_1$ (since it is a cycle, so the last element will go back to $j_1$) and the green dots represent the other digits in the cycle of $j_3$.
We use the following steps to calculate $\omega\sigma$:
\begin{enumerate}
	\item[\textbf{Step 1}] Restrict $\sigma=(j_1\colorbox{red}{...}j_2\colorbox{blue}{...})(j_3\colorbox{green}{...})...$ to the element $(j_1 j_2)(j_3) \in S_3$ by forgetting all digits except $j_1$,$j_2$,$j_3$ but preserving the cycle structure. Here, $S_3$ is $Aut\{j_1,j_2,j_3\}$. Denote $\bar{\sigma}=(j_1 j_2)(j_3)$.
	\item[\textbf{Step 2}]  Calculate $(j_3\text{ }j_2\text{ }j_1)\bar{\sigma}$. We have $(j_3\text{ }j_2\text{ }j_1)\bar{\sigma}=(j_1)(j_2 j_3)$.
	\item[\textbf{Step 3}]  Insert all numbers forgotten in the first step into $\omega\bar{\sigma}$, then we get,
	\begin{align*}
		\omega\sigma=(j_1\colorbox{red}{...})(j_2\colorbox{green}{...}j_3\colorbox{blue}{...})... \text{ .}
	\end{align*}
\end{enumerate}
Indeed, this procedure works for all cases. In conclusion, if $\omega=(j_3 \text{ } j_2 \text{ } j_1)$, we have
\begin{align*}
   (1) \quad & \sigma=(j_1...)(j_2...)(j_3...)... & \longrightarrow & \quad \omega\sigma=(j_3...j_2...j_1...)... \text{  ,}\\
   (2) \quad& \sigma=(j_1...j_2...)(j_3...)... & \longrightarrow & \quad \omega\sigma=(j_1...)(j_2...j_3...)... \text{  ,}\\
   (3) \quad& \sigma=(j_1...j_2...j_3...)... & \longrightarrow & \quad\omega\sigma=(j_1...)(j_2...)(j_3...)... \text{  ,}\\
   (4) \quad& \sigma=(j_1...j_3...j_2...)... & \longrightarrow & \quad \omega\sigma=(j_1...j_2...j_3...)... \text{  .}\\
\end{align*}
\end{construction}

\begin{example}
We consider an easy example to see how the procedure works. Let $\omega=(3 \text{ } 2 \text{ } 1)$ and $\sigma=(1 \text{ } 2 \text{ } 4)(3 \text{ }5 \text{ }6)$, so $j_i=i$, $1 \leq i \leq 3$. Then $\bar{\sigma}=(12)(3)$. We have $\omega\bar{\sigma}=(1)(23)$. So, $\omega\sigma=(1)(2 \text{ }4\text{ }3\text{ }5\text{ }6)$.
\end{example}
This process can be generalized to the product of a $d$-cycle with any permutation $\sigma$ in $S_n$.

\begin{corollary}\label{405}
Fix an integer $n$ and an arbitrary permutation $\sigma=\rho_1...\rho_l \in S_n$. We take a $3$-cycle $\omega$,
\begin{enumerate}
		\item if $\sigma$ is of type (1) with respect to $\omega$, then the number of disjoint cycles of $\omega\sigma$ is $l-2$,
		\item if $\sigma$ is of type (2) with respect to $\omega$, then the number of disjoint cycles of $\omega\sigma$ is $l$,
		\item if $\sigma$ is of type (3) with respect to $\omega$, then the number of disjoint cycles of $\omega\sigma$ is $l+2$,
		\item If $\sigma$ is of type (4) with respect to $\omega$, then the number of disjoint cycles of $\omega\sigma$ is $l$.
\end{enumerate}
\end{corollary}

\begin{proof}
It is a consequence of the calculations in Construction \ref{403}.
\end{proof}

\begin{remark}\label{406}
We make a brief review about Lemma \ref{202}, in which $\omega$ is a transposition. If $\omega=(j_2 \text{ } j_1)$, then
\begin{enumerate}
\item if $\sigma=(j_1...j_2...)...$, then $\omega\sigma=(j_1...)(j_2...)...$, (cut-operator)
\item if $\sigma=(j_1...)(j_2...)...$, then $\omega\sigma=(j_1...j_2...)... \text{ .}$ (join-operator)
\end{enumerate}
The cut operator increases the number of disjoint cycles by one and the join operator decreases the number of disjoint cycles by one. $\omega$ being a joint or cut operator is relative to $\sigma$.
\end{remark}

We want to calculate $\mu^3(\sigma)$ for arbitrary $\sigma \in S_n$. First we calculate two special cases.

\begin{lemma}\label{407}
If $n$ is odd, let $\sigma=(n \text{ } n-1 \text{ } n-2 \text{ }... \text{ } 2 \text{ }1 \text{ })$. Then, we have
\begin{align*}
\mu^3(\sigma)=\frac{n-1}{2}.
\end{align*}
If $n$ is even, let $\tilde{\sigma}=( \text{ } n \text{ } n-1 \text{ } ) \text{ } (\text{ } n-2 \text{ } n-3 \text{ }... \text{ } 2 \text{ }1 )$. Then, we have
\begin{align*}
\mu^3(\tilde{\sigma})=\frac{n}{2}.
\end{align*}
\end{lemma}

\begin{proof}
If $n$ is odd, we can use $\frac{n-1}{2}$ $3$-cycles $\delta_i$, $1 \leq i \leq \frac{n-1}{2}$, to cover $\{1,...,n\}$, i.e. the group generated by $\{\delta_i$, $1 \leq i \leq \frac{n-1}{2}\}$ is transitive on $\{1,...,n\}$,
\begin{align*}
\delta_{\frac{n-1}{2}-i+1}=(2i+1 \text{ } 2i \text{ }2i-1), \quad 1 \leq i \leq \frac{n-1}{2}.
\end{align*}
The product of these $3$-cycles is
\begin{align*}
\sigma=\delta_1...\delta_{\frac{n-1}{2}}=(n \text{ } n-1 \text{ } n-2 \text{ }... \text{ } 2 \text{ }1 \text{ }).
\end{align*}
So we can multiply $\frac{n-1}{2}$ $3$-cycles to get $\sigma$. We have to show $\frac{n-1}{2}$ is the smallest number. In Lemma \ref{201}, we know we have to use at least $n-1$ transpositions to construct $\sigma$. Each $3$-cycle can be considered as the product of two transpositions. Hence, $\frac{n-1}{2}$ is the smallest number of $3$-cycles for $\sigma$ which satisfies the condition (1),(2),(3) in Definition \ref{401}.

Similarly, if $n$ is even, the $3$-cycles we choose are
\begin{align*}
& \delta_1=( \text{ } n \text{ } n-1 \text{ } n-2 \text{ } ) , \\
& \delta_{\frac{n}{2}-i+2}=( \text{ }2i-1 \text{ } 2i-2 \text{ } 2i-3\text{ }), \quad 2 \leq i \leq \frac{n}{2}.
\end{align*}
The product of these $3$-cycles is
\begin{align*}
\tilde{\sigma}=( \text{ } n \text{ } n-1 \text{ } ) \text{ } (\text{ } n-2 \text{ } n-3 \text{ }... \text{ } 2 \text{ }1 \text{ }).
\end{align*}
With a similar argument as in the first case, we have $\mu^3(\tilde{\sigma})=\frac{n}{2}$.
\end{proof}

\begin{remark}\label{4008}
The first case of Lemma \ref{407} tells us when $n$ is odd, we have $\mu^3(\sigma)=\frac{n}{2}$ for any $n$-cycle $\sigma$. This statement is easy to check.

The second case of the lemma works for any $\tilde{\sigma}$ which is the product of two disjoint cycles (not necessary one being a transposition). It means that when $n$ is even and $\tilde{\sigma}$ is a product of arbitrary two disjoint cycles, then we have $\mu^3(\sigma)=\frac{n}{2}$. We leave it as an exercise for the reader.
\end{remark}

\begin{lemma}\label{408}
Let $n$ be a positive integer and let $\sigma$ be a permutation in the alternating group $A_n$, i.e. $\sigma$ is a product of $3$-cycles. We decompose $\sigma=\rho_1...\rho_l$ in $l$ disjoint cycles $\rho_i$, then we have
\begin{enumerate}
	\item[(1)]  if $n$ is odd, $\mu^3(\sigma)=\frac{n-1}{2}+\frac{l-1}{2}$,
	\item[(2)]  if $n$ is even, $\mu^3(\sigma)=\frac{n}{2}+\frac{l-2}{2}$.
\end{enumerate}
In conclusion, we have $\mu^3(\sigma)=\frac{n+l-2}{2}$.
\end{lemma}

\begin{proof}
First, we consider the case $n$ is odd and $\sigma \in A_n$. Since $\sigma \in A_n$ is a product of even number transpositions and we know $\mu^2(\sigma)=n+l-2$ by Lemma \ref{201}, so $l$ is an odd number. We will prove the first case by induction on $l$. When $l=1$, it means $\sigma$ is an $n$-cycle. By Lemma \ref{407} and Remark \ref{4008}, we can find $\frac{n-1}{2}$ $3$-cycles such that their products is the $n$-cycle $\sigma$. Hence, the basic step is true. Assume it is true when $l=2k-1$, we will show it is true for $l=2k+1$. We can assume $\sigma$ is a product of at least $3$ disjoint cycles in the induction step, because $l \geq 3$. Say
\begin{align*}
\sigma=\rho_1 \rho_2 \rho_3 \rho_4... \rho_{2k+1}=(j_1...)(j_2...)(j_3...)\rho_4... \rho_{2k+1}.
\end{align*}
where $\rho_1...\rho_{2k+1}$ is the product of disjoint cycles of $\sigma$.
Consider
\begin{align*}
\sigma'=(j_1...j_2...j_3...)\rho_4...\rho_{2k+1},
\end{align*}
which connects the first three disjoint cycles of $\sigma$ and has $2k-1$ disjoint cycles. By induction, we can find $\mu^3(\sigma')=\frac{n+(2k-1)-2}{2}$ $3$-cycles $\delta_1,...,\delta_{\mu^3(\sigma')}$ such that their product is $\sigma'$ and they satisfy the condition (1),(2),(3) in Definition \ref{401}. By Case (5) in Construction \ref{403} and Corollary \ref{405}, we can get $\sigma$ from $\sigma'$ by multiplying a $3$-cycle $\omega=(j_3 \text{ } j_2 \text{ } j_1)$, i.e.
\begin{align*}
\sigma=(j_3 \text{ } j_2 \text{ } j_1)\sigma'.
\end{align*}
Hence, we find $\frac{n+(2k-1)-2}{2}+1=\frac{n+(2k+1)-2}{2}$ $3$-cycles $\omega,\delta_1,...,\delta_{\mu^3(\sigma')}$ such that their product is $\sigma$. It is easy to check the group generated by $\{\omega,\delta_1,...,\delta_{\mu^3(\sigma')}\}$ is transitive on the set $\{1,...,n\}$. The only thing we have to show is the number $\frac{n+(2k+1)-2}{2}$ is minimal.
By Lemma \ref{201}, if we want to use transpositions to construct the permutation $\sigma$, we have to use $n+l-2$ transpositions. A $3$-cycle is a product of two transpositions. Hence, $\frac{n-1}{2}+\frac{l-1}{2}$ is the minimum.

The same argument holds for the case $n$ is even.
\end{proof}

\begin{remark}\label{409}
$\mu^3(\sigma)$ in Lemma \ref{408} can be computed by the Riemann-Hurwitz formula. Let $X,Y$ be two Riemann surfaces with genus $g(X)$ and $g(Y)$ and $\phi:X \rightarrow Y$ is a ramified covering map with degree $N$. Let $p$ be any point in $X$, $e_p$ is the ramification index at the point $p$. Then, the Riemann-Hurwitz formula for the covering map $\phi$ is
\begin{align*}
2g(X)-2=N(2g(Y)-2)+\sum_{p\in X}(e_p -1).
\end{align*}
Now specialize this formula to the following case : $X=Y=\mathbb{C}P^1$ and a connected ramified $n$-covering $\phi:X \rightarrow Y$ with branch points $\{z_k,...,z_1,z_{k+1}\}$ such that $z_i$ corresponds to $3$-cycles $\delta_i$, $1 \leq i \leq k$, and $z_{k+1}$ corresponds to the permutation $\sigma$. The ramification index at $z_i$ is $3$, $1 \leq i \leq k$, and the ramification index at $z_{k+1}$ is $n-l$. In this case, Riemman-Hurwitz formula says
\begin{align*}
2g(\mathbb{C}P^1)-2=n(2g(\mathbb{C}P^1)-2)+\left( n-l+(3-1)k\right).
\end{align*}
From the above formula, we get
\begin{align*}
k=\frac{n+l-2}{2},
\end{align*}
which is exactly $\mu^3(\sigma)$ we calculate in Lemma \ref{408}.
\end{remark}

\section{Generating Function}

\begin{definition}\label{501}
Let $d$ be a positive integer. We define the generating function for minimal $d$-Hurwitz numbers $h^{d}(\alpha)$ in the following way:
\begin{align*}
F_d(z,p_1,p_2,...)=\sum_{n \geq 1}\sum_{\alpha \vdash n}h^{d}(\alpha)\frac{z^n}{n!}\frac{1}{\mu^d(\alpha)!}p_\alpha.
\end{align*}
The notations are similar to those used in the generating function in Lemma \ref{203}. Note that $p_\alpha=\Phi(\alpha)$ by the map in Definition \ref{307}, so the generating function can also be written as
\begin{align*}
F_d(z,p_1,p_2,...)=\sum_{n \geq 1}\sum_{\alpha \vdash n}h^{d}(\alpha)\frac{z^n}{n!}\frac{1}{\mu^d(\alpha)!}\Phi(\alpha).
\end{align*}
\end{definition}

\begin{remark}\label{5000001}
We want to point out that in Lemma \ref{408}, we assume $\sigma$ is a permutation in the alternating group, which means $\sigma$ can always be decomposed as the product of $3$-cycles. If $\sigma$ is not in the alternating group, then $\mu^3(\sigma)$ does not make sense in this case. Indeed, we extend the definition of $\mu^3(\sigma)$ to any permutation $\sigma \in S_n$ as
\begin{align*}
\mu^3(\sigma):=\frac{n+l-2}{2}.
\end{align*}
This change of definition does not change the generating function $F_3$. Because $h^3(\alpha)=0$ if and only if $\mu^3(\alpha)$ does not make sense (see Remark \ref{4002}).
\end{remark}

The aim of this section is to derive a differential equation satisfied by $F_3$ (see Theorem \ref{505}).
\begin{definition}\label{50001}
Let $(\delta_1,...,\delta_k)$ be a $k$-tuple of $d$-cycles in $S_n$ and $\sigma=\delta_1...\delta_k$. We say $(\delta_1,...,\delta_k)$ is a minimal transitive factorization of $\sigma$, if $(\delta_1,...,\delta_k)$ satisfies the condition (2),(3),(4) in Definition \ref{401}. Since $\sigma$ is uniquely determined by $\delta_1,...,\delta_k$, sometimes we say $(\delta_1,...,\delta_k)$ is a (ordered) minimal transitive factorization. Here, "factorization" corresponds to the condition (2), "transitive" corresponds to the condition (3) and "minimal" corresponds to the condition (4).
\end{definition}

\begin{definition}\label{50002}
Let $(\delta_1,...,\delta_k)$ be a $k$-tuple, let $\delta_i \in S_n$ be $d$-cycles and let $\mathcal{S}=\{\delta_1,...,\delta_k\}$ be the corresponding set, $\sigma=\delta_1...\delta_k$. Let $G$ be the group generated by the permutations in $\mathcal{S}$. Let $X_1,...,X_q$ be the connected components of $X=\{1,...,n\}$ with respect to the action of $G$. For each connected component $X_i$, we define the subset $\mathcal{S}_i$ of $\mathcal{S}$ as
\begin{align*}
\mathcal{S}_i=\{\delta \in \mathcal{S} \mid \delta(j) \neq j \text{ for some } j \in X_i\}.
\end{align*}
Denote by $\sigma_i$ the product of the elements in $\mathcal{S}_i$ with respect to the order of the tuple $(\delta_1,...,\delta_k)$. Then, we say that the component $X_i$ corresponds to a transitive factorization of $\sigma_i$. If the elements in $\mathcal{S}_i$ satisfy Condition (3) and (4) in Definition \ref{401} on the set $X_i$, then we say $\mathcal{S}_i$ corresponds to an ordered minimal transitive factorization of $\sigma_i$.
\end{definition}

From now on we restrict to the case of $d=3$.

\begin{lemma}\label{502}
Given any integer $n$ and any permutation $\sigma$ in the alternating group $A_n$ (so that $\sigma$ has a factorization in $3$-cycles), we write
\begin{align*}
\sigma=\rho_1...\rho_l=\delta_1...\delta_{\mu^3(\sigma)},
\end{align*}
where $\rho_1...\rho_l$ is the decomposition of $\sigma$ into disjoint cycles (unique up to reordering) and $\delta_1...\delta_{\mu^3(\sigma)}$ is a product of $3$-cycles, such that the group generated by $\{\delta_i, 1 \leq i \leq \mu^3(\sigma)\}$ is transitive on $\{1,...,n\}$. Say $\delta_1=(j_3 \text{ } j_2 \text{ } j_1)$. If we consider the permutation $\sigma'=\delta_2...\delta_{\mu^3(\sigma)}$, we have the following result with respect to the four cases in Construction \ref{403}:
\begin{enumerate}
    \item if $\sigma=(j_3...j_2...j_1...)\rho_2...\rho_l$, then $\sigma'=(j_1...)(j_2...)(j_3...)\rho_2...\rho_l$. The set $\{1,...,n\}$ has three connected components with respect to the action of the group generated by $\{\delta_2,...,\delta_{\mu^3(\sigma)}\}$, each of which contains one and only one $j_i$. Each connected component corresponds to an ordered minimal transitive factorization;
	\item if $\sigma=(j_1...)(j_2...j_3...)\rho_3...\rho_l$, then $\sigma'=(j_1...j_2...)(j_3...)\rho_3...\rho_l$. The set $\{1,...,n\}$ has two connected components with respect to the action of the group generated by $\{\delta_2,...,\delta_{\mu^3(\sigma)}\}$, one contains $j_3$ and the other contains $j_1$ and $j_2$. Each connected component corresponds to an ordered minimal transitive factorization;
	\item $\sigma=(j_1...)(j_2...)(j_3...)\rho_4...\rho_l$, then $\sigma'=(j_1...j_2...j_3...)\rho_4...\rho_l$. The set $\{1,...,n\}$ is still connected with respect to the group generated by $\{\delta_2,...,\delta_{\mu^3(\sigma)}\}$;
	\item $\sigma=(j_1...j_2...j_3...)\rho_2...\rho_l$, then $\sigma'=(j_1...j_3...j_2...)\rho_2...\rho_l$. But, this case cannot happen.
\end{enumerate}
\end{lemma}

\begin{proof}
In this proof, we consider the $3$-cycle $\delta_1$ as the product of two transpositions.
\begin{enumerate}
\item[-] For Case (1), we write $\delta_1$ as the product of transpositions
\begin{align*}
\delta_1=(j_3 \text{ } j_2)(j_2 \text{ } j_1).
\end{align*}
In this case, both $(j_3 \text{ } j_2)$ and $(j_2 \text{ } j_1)$ are join-operators, more precisely, $(j_2 \text{ } j_1)$ is the join operator relative to $\sigma'$ and $(j_3 \text{ } j_2)$ is the joint operator relative to $(j_2 \text{ } j_1)\sigma'$. So, by Lemma \ref{202}, the set $\{1,...,n\}$ has three connected component with respect to the action of the group generated by $\{\delta_i,2 \leq i \leq \mu^3(\sigma)\}$. The last statement follows by an easy calculation of the $3$-cycles. We leave it to the reader to check.
\item[-] For Case (2), we consider $\delta_1=(j_1\text{ }j_3)(j_2\text{ }j_3)$. One transposition is cut-operator and the other is join-operator. With a similar discussion as in Case (1), we get the consequence.
\item[-] In Case (3), both transpositions $(j_3 \text{ } j_2),(j_2 \text{ } j_1)$ are cut-operators.
\item[-] Case (4) cannot happen, because permutations of the same type $\alpha$ have the same minimal value $\mu^3(\alpha)$. In this case, $\sigma'$ is of the same type as $\sigma$. Hence, $\mu^3(\sigma')=\mu^3(\sigma)$. But, $\delta_1\sigma'=\sigma$, it is a contradiction with the minimality of $\sigma=\delta_1...\delta_{\mu^3(\sigma)}$. Hence, this case cannot happen.
\end{enumerate}
\end{proof}

\begin{definition}\label{5001}
Given any permutation $\sigma \in A_n$ and any minimal transitive factorization $(\delta_1,...,\delta_k)$ of $\sigma$ in $3$-cycles, let $\sigma'=\delta_1^{-1}\sigma$. We say $(\sigma,\delta_1)$ or $(\sigma,\sigma')$ is of type $i$, if $\sigma$ and $\delta_1$ corresponds to Case (i) in Lemma \ref{502}, $1 \leq i \leq 3$.
\end{definition}

\begin{definition}\label{5002}
Given a positive integer $n$, let $\alpha$ be a partition of $n$. We define
\begin{align*}
& \mathcal{A}^3 (\alpha)=\{(\delta_1,...,\delta_k,\sigma) \mid \sigma \text{ is of type } \alpha,
(\delta_1,...,\delta_k) \\ & \text{ is a minimal transitive factorization of } \sigma \text{ in 3-cycles}\},\\
& \tilde{\mathcal{A}}^3(\alpha)=\{(\delta_2,...,\delta_k,\sigma) \mid (\sigma\delta_k^{-1}...\delta_2^{-1},\delta_2,...,\delta_k,\sigma) \in \mathcal{A}^3(\alpha)\}.
\end{align*}
\end{definition}

Of course, there is an obvious bijection between $\mathcal{A}^3 (\alpha)$ and $\tilde{\mathcal{A}}^3(\alpha)$:
\begin{align*}
(\delta_1,...,\delta_k,\sigma) \leftrightarrow (\delta_2,...,\delta_k,\sigma).
\end{align*}

\begin{definition}
The subset $\mathcal{A}^3_i (\alpha)$ of $\mathcal{A}^3 (\alpha)$, $1 \leq i \leq 3$, is defined as
\begin{align*}
\mathcal{A}^3_i (\alpha)=\{(\delta_1,...,\delta_k,\sigma) \in \mathcal{A}^3 (\alpha) \mid (\sigma,\delta_1) \text{ is of type i}\}.
\end{align*}
We define the subset $\tilde{\mathcal{A}}^3_i(\alpha)$ of $\tilde{\mathcal{A}}^3(\alpha)$ similarly.
\end{definition}

\begin{remark}\label{5003}
By the definition of $h^3(\alpha)$, we have
\begin{align*}
h^3(\alpha)=|\mathcal{A}^3(\alpha)|=|\tilde{\mathcal{A}}^3(\alpha)|.
\end{align*}
Also, we have disjoint unions
\begin{align*}
\mathcal{A}^3(\alpha)=\bigcup_{i=1}^{3}\mathcal{A}^3_i(\alpha), \quad \tilde{\mathcal{A}}^3(\alpha)=\bigcup_{i=1}^{3}\tilde{\mathcal{A}}^3_i(\alpha),
\end{align*}
and hence,
\begin{align*}
|\mathcal{A}^3(\alpha)|=\sum_{i=1}^3 |\mathcal{A}^3_i(\alpha)|, \quad |\tilde{\mathcal{A}}^3(\alpha)|=\sum_{i=1}^3 |\tilde{\mathcal{A}}^3_i(\alpha)|.
\end{align*}
Hence, we can write the generating function $F_3(z,p_1,p_2,...)$ as
\begin{align*}
F_3=\sum_{i=1}^3 (F_3)_i,
\end{align*}
where
\begin{align*}
(F_3)_i=\sum_{n \geq 1}\sum_{\alpha \vdash n}|\mathcal{A}^3_i(\alpha)|\frac{z^n}{n!}\frac{1}{\mu^3(\alpha)!}p_\alpha.
\end{align*}
\end{remark}

\begin{definition}
Let $\alpha'$ be a partition of $n$. We define another two type of sets
\begin{align*}
& \mathcal{B}^3_i(\alpha')=\{(\delta_2,...,\delta_k,\sigma') \mid \sigma'=\delta_2...\delta_k, \sigma' \text{ is of type } \alpha' \\ & \text{ and } (\delta_2,...,\delta_k,\sigma) \in \tilde{\mathcal{A}}^3_i(\alpha) \text{ for some permutation } \sigma \},\\
& \mathcal{B}^3_i(\alpha',\alpha)=\{(\delta_2,...,\delta_k,\sigma') \mid (\delta_2,...,\delta_k,\sigma') \in \mathcal{B}^3_i(\alpha') \text{ and } \\ & (\delta_2,...,\delta_k,\sigma) \in \tilde{\mathcal{A}}^3_i(\alpha) \text{ for some permutation } \sigma \text{ of type } \alpha\},
\end{align*}
where $1 \leq i \leq 3$.
\end{definition}

\begin{remark}
We have
\begin{align*}
\mathcal{B}_i^3(\alpha')\bigcap\mathcal{B}_j^3(\alpha')=\emptyset, \quad i \neq j, 1 \leq i,j \leq 3.
\end{align*}
This follows from Lemma \ref{502}. Indeed, fix any $k$-tuple $(\delta_2,...,\delta_k,\sigma')$, $\sigma'=\delta_2...\delta_k$, such that $(\delta_2,...,\delta_k,\sigma) \in \mathcal{A}^3(\alpha)$ for $\sigma=\delta_1\sigma'$ of type $\alpha$. If the vertex set has $3$ connected components with respect to the group $\{\delta_2,...,\delta_k\}$, then $(\delta_2,...,\delta_k,\sigma') \in \mathcal{B}_1^3(\alpha')$. If its vertex set has two connected components, then $(\delta_2,...,\delta_k,\sigma') \in \mathcal{B}_2^3(\alpha')$. If its vertex set has only one connected component, then $(\delta_2,...,\delta_k,\sigma') \in \mathcal{B}_3^3(\alpha')$.
\end{remark}

\begin{remark}
One way of thinking about $F_3$ is that it can be obtained from a "set valued" generating functions in the following way. Given a specific set $\mathcal{A}^3(\alpha)$, $\alpha \vdash n$, the elements in this set are $(k+1)$-tuples $(\delta_1,...,\delta_k,\sigma)$. The parameter corresponding to this set is $\frac{z^n}{n!}\frac{1}{\mu^3(\alpha)!}p_\alpha$, where $z$ corresponds to the integer $n$, $\mu^3(\alpha)$ corresponds to the number of $3$-cycles $k$ and $p_\alpha=\Phi(\alpha)$ corresponds to the permutation $\sigma$. We take the sum over all partitions. We will get the "set-valued" generating function
\begin{align*}
\sum_{n \geq 1}\sum_{\alpha \vdash n} \mathcal{A}^3(\alpha) \frac{z^n}{n!}\frac{1}{\mu^3(\alpha)!}p_\alpha.
\end{align*}
Since every set is finite, we can take the cardinality of each set, and we get the generating function $F_3$ in Definition \ref{501}.

In this section, we define several generating functions for the cardinalities of different sets similarly, for example the generating functions $\widetilde{F}_3$ for $|\tilde{\mathcal{A}}^3(\alpha)|$ in Definition \ref{5004} and $(\bar{F}_3)_i$ for $| \mathcal{B}^3_i(\alpha')|$ in Definition \ref{5006}.
\end{remark}

\begin{definition}\label{5004}
We define another generating function
\begin{align*}
\widetilde{F}_3(z,u,p_1,p_2,...)=\sum_{n \geq 1}\sum_{\alpha\vdash n}|\mathcal{A}^3(\alpha)|\frac{z^n}{n!}\frac{u^{\mu^3(\alpha)}}{\mu^3(\alpha)!}p_\alpha.
\end{align*}
\end{definition}
We add another parameter $u$ compared to the generating function $F_3$. Here, the exponent of $u$ indicates the number of $3$-cycles (not including $\sigma$) in $(\delta_1,...,\delta_k,\sigma) \in \mathcal{A}^3(\alpha)$. Similar to the definition of $(F_3)_i$ in Remark \ref{5003}, we define $(\widetilde{F}_3)_i$ and
\begin{align*}
\widetilde{F}_3=\sum_{i=1}^3 (\widetilde{F}_3)_i.
\end{align*}
\begin{remark}\label{5005}
Since the number of $3$-cycles in elements in $\tilde{\mathcal{A}}^3(\alpha)$ is $\mu^3(\alpha)-1$, it is natural to interpret the series $\frac{\partial \widetilde{F}_3}{\partial u}$ as generating series for $|\tilde{\mathcal{A}}^3(\alpha)|$,
\begin{align*}
\frac{\partial \widetilde{F}_3}{\partial u}=\sum_{n \geq 1}\sum_{\alpha\vdash n}|\tilde{\mathcal{A}}^3(\alpha)|\frac{z^n}{n!}\frac{u^{\mu^3(\alpha)-1}}{(\mu^3(\alpha)-1)!}p_\alpha.
\end{align*}
Similarly, we define the generating series $(\frac{\partial \widetilde{F}_3}{\partial u})_i$ for  $|\tilde{\mathcal{A}}^3_i(\alpha)|$.
\end{remark}

\begin{lemma}\label{50005}
Let $(\delta_2,...,\delta_k,\sigma')$ be a $k$-tuple of permutations. Then, we have\\
(1) if $(\delta_2,...,\delta_k,\sigma') \in \mathcal{B}^3_1(\alpha')$, then $k-1=\mu^3(\sigma')-2$,\\
(2) if $(\delta_2,...,\delta_k,\sigma') \in \mathcal{B}^3_2(\alpha')$, then $k-1=\mu^3(\sigma')-1$,\\
(3) if $(\delta_2,...,\delta_k,\sigma') \in \mathcal{B}^3_3(\alpha')$, then $k-1=\mu^3(\sigma')$.
\end{lemma}

\begin{proof}
We only give the proof of the first case. If $(\delta_2,...,\delta_k,\sigma') \in \mathcal{B}^3_1(\alpha')$, it should correspond to some $k$-tuple $(\delta_2,...,\delta_k,\sigma)$ such that $(\sigma,\sigma')$ is of type one (See Definition \ref{5001}). Hence, $\mu^3(\sigma)=k$. By Lemma \ref{408} and Case (1) in Lemma \ref{502}, we know $\mu^3(\sigma')=\mu^3(\sigma)+1$, because $\sigma'$ has two more disjoint cycles than $\sigma$. Hence, we have
\begin{align*}
\mu^3(\sigma')-2=\mu^3(\sigma)+1-2=k-1.
\end{align*}
\end{proof}

We give the following definition of the generating functions for $|\mathcal{B}_i^3(\alpha')|$ based on the above lemma. Again, the exponent of $u$ indicates the number of $3$-cycles in the $k$-tuple, which is an element in the set $\mathcal{B}^3_i(\alpha')$.

\begin{definition}\label{5006}
We define the generating function  $(\bar{F}_3)_1$ for $|\mathcal{B}_1^3(\alpha')|$ as
\begin{align*}
(\bar{F}_3)_1=\sum_{n \geq 1}\sum_{\alpha' \vdash n}|\mathcal{B}_1^3(\alpha')|\frac{z^n}{n!}\frac{u^{\mu^3(\alpha')-2}}{(\mu^3(\alpha')-2)!}p_{\alpha'}.
\end{align*}
Similarly, we can define the generating function for $|\mathcal{B}_i^3(\alpha')|$, $i=2 \text{ or }3$,
\begin{align*}
(\bar{F}_3)_2=\sum_{n \geq 1}\sum_{\alpha' \vdash n}|\mathcal{B}_2^3(\alpha')|\frac{z^n}{n!}\frac{u^{\mu^3(\alpha')-1}}{(\mu^3(\alpha')-1)!}p_{\alpha'},
\end{align*}
and
\begin{align*}
(\bar{F}_3)_3=\sum_{n \geq 1}\sum_{\alpha' \vdash n}|\mathcal{B}_3^3(\alpha')|\frac{z^n}{n!}\frac{u^{\mu^3(\alpha')}}{\mu^3(\alpha')!}p_{\alpha'}.
\end{align*}
\end{definition}

\begin{definition}
Given two partitions $\alpha \vdash n_1$, $\beta \vdash n_2$, we define $\alpha \bigcup \beta$ as the partition of $n_1+n_2$, whose parts are those of $\alpha$ and $\beta$, arranged in descending order.
\end{definition}

The next remark is about how to give an order to the connected components of $\{1,...,n\}$ with respect to the group generated by $\{\delta_2,...,\delta_k\}$.

\begin{remark}\label{500008}
Given $(\delta_2,...,\delta_k,\sigma') \in \mathcal{B}_1^3(\alpha')$, $\alpha' \vdash n$, the vertex set $\{1,...,n\}$ has three connected components with respect to the group generated by $\{\delta_2,...,\delta_k\}$. Recall the notations in Definition \ref{50002}, let $X_1,X_2,X_3$ be the three connected components with respect to the group generated by $\{\delta_2,...,\delta_k\}$, where $X_1$ is the connected component containing $1$, $X_2$ contains the smallest number in $\{1,...,n\}/X_1$ and $X_3$ is the third component. Clearly, this gives a well-defined order on the connected components of $\{1,...,n\}$. We call it the canonical order. For each permutation $\varepsilon \in S_3$, we can define a new order on the three connected components:
\begin{align*}
X_1^{\varepsilon},X_2^{\varepsilon},X_3^{\varepsilon},
\end{align*}
where $X_i^{\varepsilon}=X_{\varepsilon(i)}$.
\end{remark}

We will give some definitions which will be used in Construction \ref{50008}, Lemma \ref{5007}, \ref{5009} and \ref{5010}.

\begin{definition}
We define the following sets
\begin{align*}
& \mathcal{OB}_1^3(\alpha')=\{(\delta_2,...,\delta_k,\sigma',\varepsilon) \mid (\delta_2,...,\delta_k,\sigma') \in \mathcal{B}_1^3(\alpha'), \varepsilon \in S_3\}, \\
& \mathcal{OB}_1^3(\alpha',\alpha)=\{(\delta_2,...,\delta_k,\sigma',\varepsilon) \mid (\delta_2,...,\delta_k,\sigma') \in \mathcal{B}_1^3(\alpha',\alpha), \varepsilon \in S_3\}.
\end{align*}
\end{definition}
$\mathcal{OB}_1^3(\alpha')$ is the set of elements in $\mathcal{B}_1^3(\alpha')$ with a particular order $\varepsilon$ on the connected components of its vertex set and the same for $\mathcal{OB}_1^3(\alpha',\alpha)$.

\begin{remark}
Given any partition of $\alpha' \vdash n$, we have
\begin{align*}
\mathcal{OB}^3_i(\alpha')= \bigcup_{\alpha \vdash n} \mathcal{OB}^3_i(\alpha',\alpha).
\end{align*}
The union here is not disjoint. Indeed, consider the following example
\begin{align*}
& \delta_2=(1 \text{ } 2 \text{ } 3),\quad \delta_3=(4 \text{ } 5 \text{ } 6),\quad \delta_4=(7 \text{ } 8 \text{ } 9),\quad \delta_5=(9 \text{ } 8 \text{ } 10), \\
& \alpha'=(1 3^3), \quad \sigma'=\delta_2\delta_3\delta_4\delta_5=(1 \text{ } 2 \text{ } 3)(4 \text{ } 5 \text{ } 6)(7 \text{ } 8 \text{ } 10)(9).
\end{align*}
Clearly, $(\delta_2,\delta_3,\delta_4,\delta_5,\sigma') \in \mathcal{OB}^3_1(\alpha')$. $\{1,...,10\}$ has three connected component with respect to the group generated by $\{\delta_2,...,\delta_5\}$. They are
\begin{align*}
X_1=\{1,2,3\},\quad X_2=\{4,5,6\},\quad X_3=\{7,8,9,10\}.
\end{align*}
Let $\delta_1=(1 \text{ } 4 \text{ }7)$ and $\tilde{\delta}_1=(1 \text{ } 4 \text{ }9)$. Then, we have
\begin{align*}
& \sigma=\delta_1\sigma'=(1\text{ } 2 \text{ }3 \text{ }4 \text{ }5 \text{ }6 \text{ }7 \text{ }8\text{ } 10)(9),\\
& \widetilde{\sigma}=\widetilde{\delta}_1\sigma'=(1 \text{ }2 \text{ }3 \text{ }4 \text{ }5 \text{ }6 \text{ }9)(7\text{ } 8\text{ } 10).
\end{align*}
Clearly, $\sigma$ and $\tilde{\sigma}$ are of different types: $\sigma$ is of type $\alpha=(9 \text{ } 1)$ and $\widetilde{\sigma}$ is of type $\widetilde{\alpha}=(7 \text{ } 3)$. Hence,
\begin{align*}
(\delta_2,\delta_3,\delta_4,\delta_5,\sigma') \in \mathcal{OB}^3_1(\alpha',\alpha) \bigcap\mathcal{OB}^3_1(\alpha',\widetilde{\alpha}).
\end{align*}
\end{remark}

\begin{lemma}\label{5007}
We have
\begin{align*}
& (\bar{F}_3)_1=\frac{1}{3!}(\tilde{F}_3)^3,\\
& (\bar{F}_3)_2=\frac{1}{2!}(\tilde{F}_3)^2, \\
& (\bar{F}_3)_3=\sum_{n \geq 1}\sum_{\alpha \vdash n, \atop \text{ at least one } \alpha_i \geq 3}h^{3}(\alpha)\frac{z^n}{n!}\frac{u^{\mu^3(\alpha)}}{\mu^3(\alpha)!} p_\alpha.
\end{align*}
\end{lemma}

\begin{proof}
We only give the proof for the first equation. Given $(\delta_2,...,\delta_k,\sigma',\varepsilon) \in \mathcal{OB}_1^3(\alpha')$, $\alpha' \vdash n$, the vertex set has three connected components with respect to the group generated by $\{\delta_2,...,\delta_k\}$, called $X_1,X_2,X_3$. The order of these connected components is determined by $\varepsilon$ and it is $X_{\varepsilon(1)},X_{\varepsilon(2)},X_{\varepsilon(3)}$ (see Remark \ref{500008}). Define $\mathcal{S}'=\{\delta_2,...,\delta_k\}$. $\mathcal{S}'_i$ is the subset of $\mathcal{S}'$, which corresponds to the connected component $X_i$, and $\sigma_i$ is the product of all permutations in $\mathcal{S}'_i$ with respect to the order of the tuple $(\delta_2,...,\delta_k)$. Denote by $\alpha_i$ the type of $\sigma_i$, $\alpha_i \vdash n_i$, $1 \leq i \leq 3$.

Now let's take one $(\mu^3(\alpha_i)+1)$-tuple $(\delta_1^i,...,\delta_{\mu^3(\alpha_i)}^i,\hat{\sigma}_i)$ from each set $\mathcal{A}_1^3(\alpha_i)$, $1 \leq i \leq 3$, where $\alpha_1\bigcup\alpha_2\bigcup\alpha_3=\alpha'$. We want to put them together to construct a tuple in $\mathcal{OB}_1^3(\alpha')$. First, we have to fix the vertex sets for each minimal transitive factorization. It means we have to choose $n_i$ integers from $\{1,...,n\}$ as the vertex set for $\mathcal{A}_1^3(\alpha_i)$. The number of choices for the vertex sets is
\begin{align*}
{n \choose n_1}{n-n_1 \choose n_2}{n-n_1-n_2\choose n_3}.
\end{align*}
Next, we have to fix the order of $3$-cycles in the new tuple. In the new tuple, there are $\sum_{i=1}^3\mu^3(\alpha_i)$ many $3$-cycles. So, $\mu^3(\alpha_i)$ of them are for the $3$-cycles in $(\delta_1^i,...,\delta_{\mu^3(\alpha_i)}^i,\hat{\sigma}_i)$. So, the number of choices for the positions is
\begin{align*}
{\sum_{i=1}^3\mu^3(\alpha_i) \choose \mu^3(\alpha_1)}{\sum_{i=1}^3\mu^3(\alpha_i)-\mu^3(\alpha_1) \choose \mu^3(\alpha_2)}{\sum_{i=1}^3\mu^3(\alpha_i)-\mu^3(\alpha_1)-\mu^3(\alpha_2) \choose \mu^3(\alpha_3)}.
\end{align*}
After we choose $\mu^3(\alpha_i)$ positions for the $3$-cycles $\delta_v^i$, $1 \leq v \leq \mu^3(\alpha_i)$, the order of these $3$-cycles in the new tuple is the same as the order in $(\delta_1^i,...,\delta_{\mu^3(\alpha_i)}^i)$. With the discussion and notations above, we have
\begin{align*}
|\mathcal{OB}_1^3(\alpha')|=&\sum_{\alpha_1\bigcup\alpha_2\bigcup\alpha_3=\alpha'} ( {n \choose n_1}{\sum_{i=1}^3\mu^3(\alpha_i) \choose \mu^3(\alpha_1)}|\mathcal{A}_1^3(\alpha_1)|\times
 \\&{n-n_1 \choose n_2}{\sum_{i=1}^3\mu^3(\alpha_i)-\mu^3(\alpha_1) \choose \mu^3(\alpha_2)}|\mathcal{A}_1^3(\alpha_2)|\times
 \\ &{n-n_1-n_2\choose n_3}{\sum_{i=1}^3\mu^3(\alpha_i)-\mu^3(\alpha_1)-\mu^3(\alpha_2) \choose \mu^3(\alpha_3)}|\mathcal{A}_1^3(\alpha_3)|).
\end{align*}
Now we calculate the generating function $(\tilde{F}_3)^3$.
\begin{align*}
(\tilde{F}_3)^3 & = \prod_{i=1}^{3}\sum_{n_i \geq 1}\sum_{\alpha_i \vdash n_i}|\mathcal{A}^3(\alpha_i)|\frac{z^{n_i}}{n_i!}\frac{u^{\mu^3(\alpha_i)}}{\mu^3(\alpha_i)!}p_{\alpha_i}\\
&= \sum_{n_i \geq 1, \atop n_1+n_2+n_3=n}\sum_{\alpha_i \vdash n_i,\atop \alpha_i\bigcup\alpha_i\bigcup\alpha_i=\alpha'}|\mathcal{A}_1^3(\alpha_1)||\mathcal{A}_1^3(\alpha_2)||\mathcal{A}_1^3(\alpha_3)| \\&\frac{z^{n}}{n_1!n_2!n_3!}\frac{u^{\sum_{i=1}^3\mu^3(\alpha_i)}}{\mu^3(\alpha_1)!\mu^3(\alpha_2)!\mu^3(\alpha_3)!}p_{\alpha'}
\\&=\sum_{n \geq 1}\sum_{\alpha' \vdash n}|\mathcal{OB}^3_1(\alpha)|\frac{z^{n}}{n!}\frac{u^{\mu^3(\alpha')-2}}{(\mu^3(\alpha')-2)!}p_{\alpha'},
\end{align*}
where the last equality comes from the following formulas
\begin{align*}
 \frac{1}{n_1!n_2!n_3!}&=\frac{{n \choose n_1}{n-n_1 \choose n_2}{n-n_1-n_2\choose n_3}}{n!},\\
 \frac{1}{\mu^3(\alpha_1)!\mu^3(\alpha_2)!\mu^3(\alpha_3)!}&={\sum_{i=1}^3\mu^3(\alpha_i) \choose \mu^3(\alpha_1)}{\sum_{i=1}^3\mu^3(\alpha_i)-\mu^3(\alpha_1) \choose \mu^3(\alpha_1)}\\
 &{\sum_{i=1}^3\mu^3(\alpha_i)-\mu^3(\alpha_1)-\mu^3(\alpha_2) \choose \mu^3(\alpha_1)}\frac{1}{(\sum_{i=1}^3\mu^3(\alpha_i))!},\\
 \sum_{i=1}^3\mu^3(\alpha_i)&=\mu^3(\alpha')-2 \text{ } (\text{Lemma \ref{50005}}).
\end{align*}
By definition, we have
\begin{align*}
|\mathcal{OB}^3_1(\alpha')|=6|\mathcal{B}^3_1(\alpha')|.
\end{align*}
Thus, we prove the first equation,
\begin{align*}
(\bar{F}_3)_1=\frac{1}{6}(\tilde{F}_3)^3.
\end{align*}
With a similar argument, we can prove the other formulas.
\end{proof}

Let $\omega=(j_3 \text{ }  j_2 \text{ } j_1)$ be a $3$-cycle in $S_n$ and let $\sigma$ be a permutation in $S_n$. Let $1 \leq i \leq 3$ be a fixed integer. Let
\begin{align*}
\mathcal{L}_i=\{l \geq 1 \mid \sigma^{l}(j_i) \text{ is any }j_{k}, 1 \leq k \leq 3\}.
\end{align*}
$\mathcal{L}_i$ is not empty, since $n! \in \mathcal{L}_i$ (as $\sigma^{n!}$ is the identity map on the set $\{1,...,n\}$, and $\sigma^{n!}(j_i)=j_i$).

\begin{definition}\label{4010}
We define the "distance" between $j_i$ and the set $\{j_1,j_2,j_3\}$ with respect to the permutation $\sigma$ as
\begin{align*}
dist(j_i,\sigma,j_1,j_2,j_3) = min (\mathcal{L}_i).
\end{align*}
\end{definition}

\begin{example}
We give some examples about the definition above. Consider Case (3) in Construction \ref{403},
\begin{align*}
\omega=\delta_1=(j_3 \text{ } j_2 \text{ } j_1), \quad \sigma=(j_1...j_2...j_3...)\rho_2...\rho_l,
\end{align*}
where $\rho_1=(j_1...j_2...j_3...)$. $dist(j_3,\sigma,j_1,j_2,j_3)$ is the "distance" between $j_3$ and $j_1$ in the cycle $\rho_1$, because $j_1$ is the first element in $\{j_1.j_2,j_3\}$ after $j_3$ under the action of $\sigma$. Similarly, $dist(j_2,\sigma,j_1,j_2,j_3)$ is the "distance" between $j_2$ and $j_3$. Clearly, $\sum_{1 \leq i \leq 3}dist(j_i,\sigma,j_1,j_2,j_3)$ is the length of the cycle $\rho_1$.

Now, let's consider Case (1) in Construction \ref{403}. Here,
\begin{align*}
\sigma=(j_1...)(j_2...)(j_3...)\rho_4...\rho_l.
\end{align*}
In this case, $dist(j_i,\sigma,j_1,j_2,j_3)$ is the length of the cycle containing $j_i$.
\end{example}

\begin{remark}
Suppose $\sigma',\delta$ are permutations in $S_n$, where $\delta$ is a $3$-cycle $(j_3 \text{ } j_2\text{ } j_1)$. Let $\sigma=\delta \sigma'$. Then, we have
\begin{align*}
dist(j_i,\sigma,j_1,j_2,j_3)=dist(j_i,\sigma',j_1,j_2,j_3), \quad 1 \leq i \leq 3.
\end{align*}
This property comes from the calculation in Construction \ref{403}.
\end{remark}

Given an element $(\delta_2,...,\delta_k,\sigma) \in \widetilde{\mathcal{A}}^3_1(\alpha)$, the vertex set $\{1,...,n\}$ has three unordered connected components with respect to the action of the group generated by $\{\delta_2,...,\delta_k\}$. There are $6$ possible orders on the connected components, each of which can be represented by an element $\varepsilon \in S_3$ (Remark \ref{500008}).

\begin{definition}
We define the set $\widetilde{\mathcal{WOA}}^3_1(\alpha)$ as following
\begin{align*}
\widetilde{\mathcal{WOA}}^3_1(\alpha)=\{(\delta_2,...,\delta_k,\sigma,\varepsilon) \mid (\delta_2,...,\delta_k,\sigma) \in \tilde{\mathcal{A}}^3_1(\alpha), \varepsilon \in S_3\}.
\end{align*}
So elements of $\widetilde{\mathcal{WOA}}^3_1(\alpha)$ are elements in $\tilde{\mathcal{A}}^3_1(\alpha)$ together with an order of connected components.
\end{definition}

Now we will show how to use the elements in $\mathcal{OB}^3_1(\alpha')$ to construct a subset $\mathcal{OA}_1^3(\alpha',i_1,i_2,i_3)$ of $\widetilde{\mathcal{WOA}}^3_1(\alpha)$.

\begin{construction}\label{50008}
Let $(\delta_2,...,\delta_k,\sigma',\varepsilon)$ be an element in $\mathcal{OB}^3_1(\alpha')$. Recall the notations in Definition \ref{50002}, we define $X_1,X_2,X_3$ to be the three connected components with respect to the group generated by $\{\delta_2,...,\delta_k\}$, where the order here is the canonical order defined in Remark \ref{500008} rather than the order $\varepsilon$. Let $\mathcal{S}=\{\delta_2,...,\delta_k\}$. $\mathcal{S}_i$ is the subset of $\mathcal{S}$, which corresponds to the connected component $X_i$, and $\sigma_i$ is the product of all permutations in $\mathcal{S}_i$ with respect to the order of the tuple $(\delta_2,...,\delta_k)$. We have $\sigma'=\sigma_1\sigma_2\sigma_3$. We fix three positive integers $i_1,i_2,i_3$.

Recall in Lemma \ref{502}, we use a $3$-cycle $\delta_1$ to connect three connected components or, more precisely, to connect three disjoint cycles from each connected component. We assume there is at least one disjoint cycle with length $i_v$ in $\sigma_v$, $1 \leq v \leq 3$. We take one disjoint cycle $\rho'_v$ with length $i_v$ from $\sigma_v$ and pick one integer from each of the three cycles. Assume we take $j_v$ from the cycle $\rho'_v$, $1 \leq v \leq 3$. There are two choices of $3$-cycles after we pick the integers $j_1,j_2,j_3$. They are $(j_3 \text{ } j_2 \text{ } j_1)$ and $(j_3 \text{ } j_1 \text{ } j_2)$. We use the order $\varepsilon$ of connected components to determine which is the $3$-cycle we want. The $3$-cycle we choose is $(j_{\varepsilon(3)} \text{ } j_{\varepsilon(2)} \text{ } j_{\varepsilon(1)})$. We can use this $3$-cycle to construct a new permutation $(j_{\varepsilon(3)} \text{ } j_{\varepsilon(2)} \text{ } j_{\varepsilon(1)}) \sigma'$. Clearly, $(\delta_2,...,\delta_k,(j_{\varepsilon(3)} \text{ } j_{\varepsilon(2)} \text{ } j_{\varepsilon(1)})\sigma',\varepsilon)$ is an element in $\widetilde{\mathcal{WOA}}^3_1(\alpha)$.

Denote by $\mathcal{OA}_1^3(\alpha',i_1,i_2,i_3)$ the set of $(k+1)$-tuples $(\delta_2,...,\delta_k,\sigma,\varepsilon)$ such that $(\delta_2,...,\delta_k,\sigma,\varepsilon)$ can be constructed from some element $(\delta_2,...,\delta_k,\sigma',\varepsilon) \in \mathcal{OB}_1^3(\alpha')$ by connecting three disjoint cycles (or multiplying by a $3$-cycle $\delta_1$) with length $i_1,i_2,i_3$ from $\sigma_1,\sigma_2,\sigma_3$ in the above method. $\mathcal{OA}_1^3(\alpha',i_1,i_2,i_3)$ is a subset of $\widetilde{\mathcal{WOA}}^3_1(\alpha)$.
\end{construction}

\begin{remark}
If $\sigma_v$ does not have a disjoint cycle with length $i_v$, where $1 \leq v \leq 3$, then $\mathcal{OA}_1^3(\alpha',i_1,i_2,i_3)=\emptyset$.
\end{remark}

\begin{lemma}\label{5000010}
Assume $\mathcal{OA}_1^3(\alpha',i_1,i_2,i_3)$ is nonempty. Given any two elements $(\delta_2,...,\delta_k,\sigma,\varepsilon), (\widetilde{\delta}_2,...,\widetilde{\delta}_k,\widetilde{\sigma},\varepsilon)$ in the set $\mathcal{OA}_1^3(\alpha',i_1,i_2,i_3)$, $\sigma$ and $\widetilde{\sigma}$ are of the same type, i.e. $\Phi(\sigma)=\Phi(\widetilde{\sigma})$.
\end{lemma}

\begin{proof}
Given any two elements $(\delta_2,...,\delta_k,\sigma',\varepsilon), (\widetilde{\delta}_2,...,\widetilde{\delta}_k,\widetilde{\sigma}',\varepsilon) \in \mathcal{OB}_1^3(\alpha')$, we have $\Phi(\alpha')=\Phi(\sigma')=\Phi(\widetilde{\sigma}')$ by definition. Assume $(\delta_2,...,\delta_k,\sigma,\varepsilon)$ is constructed from the element $(\delta_2,...,\delta_k,\sigma',\varepsilon)$ by multiplying a $3$-cycle $\delta_1$ and $(\widetilde{\delta}_2,...,\widetilde{\delta}_k,\widetilde{\sigma},\varepsilon)$ is constructed from $(\widetilde{\delta}_2,...,\widetilde{\delta}_k,\widetilde{\sigma}',\varepsilon)$ by multiplying a $3$-cycle $\widetilde{\delta}_1$. Also, we write
\begin{align*}
\sigma'=\rho_1...\rho_l, \quad \widetilde{\sigma}'=\widetilde{\rho}_1...\widetilde{\rho}_l,
\end{align*}
where $\rho_1...\rho_l$ is the product of disjoint cycles of $\sigma'$ and the same for $\widetilde{\rho}_1...\widetilde{\rho}_l$ and $\widetilde{\sigma}'$. We can assume the length of $\rho_v$ equal to the length of $\widetilde{\rho}_v$. $\delta_1$ connects the first three disjoint cycles $\rho_1,\rho_2,\rho_3$ in $\sigma'$ and $\widetilde{\delta}_1$ connects the first three disjoint cycles in $\widetilde{\sigma}'$. Then, we have
\begin{align*}
\sigma=\rho'_1\rho_4...\rho_l, \quad \widetilde{\sigma}=\widetilde{\rho}'_1\widetilde{\rho}_4...\widetilde{\rho}_l,
\end{align*}
where $\rho'_1=\delta_1\rho_1\rho_2\rho_3$ and $\widetilde{\rho}'_1=\widetilde{\delta}_1\widetilde{\rho}_1\widetilde{\rho}_2\widetilde{\rho}_3$, which are disjoint cycles of the same length $i_1+i_2+i_3$. Hence, $\sigma$ and $\tilde{\sigma}$ are of the same type.
\end{proof}

\begin{remark}
Now we want to translate the proof of Lemma \ref{5000010} in differential operators. We assume that there is only one disjoint cycle with length $i_v$ in $\sigma'$, $1 \leq v \leq 3$. Multiplying a $3$-cycle $\delta_1$ to $\sigma'$ means that we substitute $\rho_1\rho_2\rho_3$ by another cycle $\rho'_1$. This procedure can be considered in two steps: we first delete the first three cycles, then add another cycle $\rho'_1$. We consider this procedure in monomial $\Phi(\sigma')$. Deleting the first three cycles means
\begin{align*}
\frac{\partial^3 \Phi(\sigma')}{\partial p_{i_1}\partial p_{i_2}\partial p_{i_3}},
\end{align*}
and adding another new cycle means
\begin{align*}
 p_{i_1+i_2+i_3}\frac{\partial^3 \Phi(\sigma')}{\partial p_{i_1}\partial p_{i_2}\partial p_{i_3}}.
\end{align*}
Hence, we have
\begin{align*}
 p_{i_1+i_2+i_3}\frac{\partial^3 \Phi(\sigma')}{\partial p_{i_1}\partial p_{i_2}\partial p_{i_3}}  = \Phi(\sigma).
\end{align*}
Similarly, we have
\begin{align*}
 p_{i_1+i_2+i_3}\frac{\partial^3 \Phi(\widetilde{\sigma}')}{\partial p_{i_1}\partial p_{i_2}\partial p_{i_3}}  = \Phi(\widetilde{\sigma}).
\end{align*}
Since $\Phi(\sigma')=\Phi(\widetilde{\sigma}')$, so $\Phi(\sigma)=\Phi(\widetilde{\sigma})$.
\end{remark}

\begin{remark}\label{5000011}
From Lemma \ref{5000010}, we know that for any element $(\delta_2,...,\delta_k,\sigma,\varepsilon)$ in $\mathcal{OA}_1^3(\alpha',i_1,i_2,i_3)$, $\sigma$ is of some fixed type $\alpha$. Sometimes, we will use the following notation to emphasize the type $\alpha$ for the set $\mathcal{OA}_1^3(\alpha',i_1,i_2,i_3)$,
\begin{align*}
\mathcal{OA}_1^3(\alpha',\alpha,i_1,i_2,i_3):=\mathcal{OA}_1^3(\alpha',i_1,i_2,i_3).
\end{align*}
For any other partition $\widetilde{\alpha} \neq \alpha$, we define
\begin{align*}
\mathcal{OA}_1^3(\alpha',\widetilde{\alpha},i_1,i_2,i_3):=\emptyset.
\end{align*}
\end{remark}

\begin{definition}\label{50009}
We define the union of all sets $\mathcal{OA}_1^3(\alpha',\alpha,i_1,i_2,i_3)$ as
\begin{align*}
\mathcal{OA}_1^3(\alpha)=\bigcup_{i_1,i_2,i_3 \geq 1}\bigcup_{\alpha'} \mathcal{OA}_1^3(\alpha',\alpha,i_1,i_2,i_3),
\end{align*}
which is a disjoint union.
\end{definition}

\begin{lemma}\label{5010}
Given any partition $\alpha\vdash n$, we have
\begin{align*}
|\widetilde{\mathcal{WOA}}^3_1(\alpha)|=2|\mathcal{OA}^3_1(\alpha)|.
\end{align*}
\end{lemma}

\begin{proof}
Recall that $\mathcal{OA}_1^3(\alpha',\alpha,i_1,i_2,i_3)$ is a subset of $\widetilde{\mathcal{WOA}}^3_1(\alpha)$ and $\mathcal{OA}^3_1(\alpha)$ is the union of the sets $\mathcal{OA}_1^3(\alpha',\alpha,i_1,i_2,i_3)$ over all partitions $\alpha'$ and all positive integers $i_1,i_2,i_3$.

Fix an element $(\delta_2,...,\delta_k,\sigma') \in \mathcal{B}^3_1(\alpha')$. There are $6$ possible orders on the connected components of $\{1,...,n\}$ with respect to the action of the group generated by $\{\delta_2,...,\delta_k\}$. Each corresponds to an element in $\mathcal{OB}^3_1(\alpha')$. In Construction \ref{50008}, after we pick three integers $j_1,j_2,j_3$, the order $\varepsilon$ determines the unique $3$-cycle. We assume $j_v$ comes from $X_v$, $1\leq v \leq 3$. The $3$-cycle we construct from $(\delta_2,...,\delta_k,\sigma',\varepsilon)$ by taking the integers $j_1,j_2,j_3$ from each connected component is $(j_{\varepsilon(3)} \text{ } j_{\varepsilon(2)} \text{ } j_{\varepsilon(1)})$, $\varepsilon \in S_3$. Now let's consider the six $3$-cycles $(j_{\varepsilon(3)} \text{ } j_{\varepsilon(2)} \text{ } j_{\varepsilon(1)})$, where $\varepsilon \in S_3$. We find three of them are the same, i.e.
\begin{align*}
(j_3 \text{ } j_2 \text{ } j_1)=(j_1 \text{ } j_3 \text{ } j_2)=(j_2 \text{ } j_1 \text{ } j_3),\\
(j_3 \text{ } j_1 \text{ } j_2)=(j_2 \text{ } j_3 \text{ } j_1)=(j_1 \text{ } j_2 \text{ } j_3).
\end{align*}
Say the $3$-cycles in the first row correspond to permutations $\varepsilon_1,\varepsilon_2,\varepsilon_3$ respectively and the other $3$-cycles in the second row correspond to permutations $\varepsilon_4,\varepsilon_5,\varepsilon_6$ respectively. We find
\begin{align*}
\sigma=\varepsilon_i \sigma'\neq \varepsilon_j \sigma' = \tilde{\sigma},
\end{align*}
where $1 \leq i \leq 3$ and $4 \leq j \leq 6$.

By the discussion above, we can only construct the elements $(\delta_2,...,\delta_k,\sigma,\varepsilon_i)$, where $i=1,2,3$, and we cannot get the elements $(\delta_2,...,\delta_k,\sigma,\varepsilon_j) \in \widetilde{\mathcal{WOA}}^3_1(\alpha)$. It means $(\delta_2,...,\delta_k,\sigma,\varepsilon_j) \notin \mathcal{OA}^3_1(\alpha)$. Also, $(\delta_2,...,\delta_k,\tilde{\sigma},\varepsilon_i) \notin \mathcal{OA}^3_1(\alpha)$, where $i=1,2,3$, and $(\delta_2,...,\delta_k,\tilde{\sigma},\varepsilon_j) \in \mathcal{OA}^3_1(\alpha)$ for $j=4,5,6$. Hence, we only construct half of the elements in $|\widetilde{\mathcal{WOA}}^3_1(\alpha)|$. Hence, we have
\begin{align*}
6|\tilde{\mathcal{A}}^3_1(\alpha)|=|\widetilde{\mathcal{WOA}}^3_1(\alpha)|=2|\mathcal{OA}^3_1(\alpha',\alpha)|.
\end{align*}
\end{proof}

The notations in the following lemma are the same as the notations in Construction \ref{50008}.
\begin{lemma}\label{5000012}
Let $(\delta_2,...,\delta_k,\sigma',\varepsilon)$ be an element in $\mathcal{OB}^3_1(\alpha')$ and let $i_1,i_2,i_3$ be positive integers. $c_v$ is the number of disjoint cycles with length $i_v$ in $\sigma_v$. The number of elements in $\mathcal{OA}_1^3(\alpha',i_1,i_2,i_3)$ which are constructed from $(\delta_2,...,\delta_k,\sigma',\varepsilon)$ is $\prod^3_{v=1}c_v i_v$.
\end{lemma}

\begin{proof}
If $c_v=0$ for some $1 \leq v \leq 3$, then $\mathcal{OA}_1^3(\alpha',i_1,i_2,i_3)$  is empty. So, the number of elements in $\mathcal{OA}_1^3(\alpha',i_1,i_2,i_3)$ is zero. Also, $\prod^3_{v=1}c_v i_v=0$. Hence, the statement is true in this special case.

Now we assume that there is at least one disjoint cycle with length $i_v$ in $\sigma_v$, $1 \leq v \leq 3$. In Construction \ref{50008}, we first pick disjoint cycle $\rho'_v$ with length $i_v$ in $\sigma_v$, $1 \leq v \leq 3$. The number of the choices of $\rho'_1,\rho'_2,\rho'_3$ is $\prod^3_{v=1}c_v$. After we pick three disjoint cycles $\rho'_1,\rho'_2,\rho'_3$, we can construct $i_1 i_2 i_3$ many $\delta_1$ such that $(\delta_2,...,\delta_k,\delta_1\sigma',\varepsilon) \in \mathcal{OA}_1^3(\alpha',i_1,i_2,i_3)$. Hence, the number of elements constructed by $(\delta_2,...,\delta_k,\sigma',\varepsilon)$ is $\prod^3_{v=1}c_v i_v$.
\end{proof}

\begin{lemma}\label{50010}
Let $i_1,i_2,i_3$ be three positive integers. We have
\begin{align*}
& \sum_{(\delta_2,...,\delta_k,\sigma',\varepsilon) \in \mathcal{OB}^3_1(\alpha')}i_1 i_2 i_3 p_{i_1+i_2+i_3} \frac{\partial \Phi(\sigma_1)}{\partial p_{i_1}}\frac{\partial \Phi(\sigma_2)}{\partial p_{i_2}}\frac{\partial \Phi(\sigma_3)}{\partial p_{i_3}}\\
& =\left( \sum_{(\delta_2,...,\delta_k,\sigma,\varepsilon) \in \mathcal{OA}^3_1(\alpha',i_1,i_2,i_3)}\Phi(\sigma) \right).
\end{align*}
\end{lemma}
\begin{proof}
Let $(\delta_2,...,\delta_k,\sigma',\varepsilon)$ be an element in $\mathcal{OB}^3_1(\alpha')$. We define a new set \begin{align*}
\mathcal{OA}^3_1(\alpha',\delta_2,...,\delta_k,\sigma',\varepsilon),
\end{align*}
which contains all elements constructed from $(\delta_2,...,\delta_k,\sigma',\varepsilon)$ as in Construction \ref{50008}. It is a subset of $\mathcal{OA}^3_1(\alpha',i_1,i_2,i_3)$. Clearly, we have
\begin{align*}
\bigcup_{(\delta_2,...,\delta_k,\sigma',\varepsilon) \in \mathcal{OB}^3_1(\alpha')}\mathcal{OA}^3_1(\alpha',\delta_2,...,\delta_k,\sigma',\varepsilon) = \mathcal{OA}^3_1(\alpha',i_1,i_2,i_3),
\end{align*}
which is a disjoint union.

Now we begin to prove this lemma. First, if we cannot find a disjoint cycle with length $i_v$ in $\sigma_v$ for some $v$, $1 \leq v \leq 3$, it means that $\mathcal{OA}^3_1(\alpha',\delta_2,...,\delta_k,\sigma',\varepsilon)$ is empty. So, we have
\begin{align*}
\left( \sum_{(\delta_2,...,\delta_k,\sigma,\varepsilon) \in \mathcal{OA}^3_1(\alpha',\delta_2,...,\delta_k,\sigma',\varepsilon)}\Phi(\sigma) \right)=0
\end{align*}
Also, we find $\frac{\partial \Phi(\sigma_v)}{\partial p_{i_v}}=0$. So, the lemma is true in this special case.

Now we assume there is at least one disjoint cycle with length $i_v$ in $\sigma_v$ for all $1 \leq v \leq 3$ and $c_v$ is the number of disjoint cycles with length $i_v$ in $\sigma_v$. By Lemma \ref{5000012}, we know the number of elements in $\mathcal{OA}_1^3(\alpha',i_1,i_2,i_3)$ which is constructed from $(\delta_2,...,\delta_k,\sigma',\varepsilon)$ is $\prod^3_{v=1}c_v i_v$. So, we have
\begin{align*}
\left( \sum_{(\delta_2,...,\delta_k,\sigma,\varepsilon) \in \mathcal{OA}^3_1(\alpha',\delta_2,...,\delta_k,\sigma',\varepsilon)}\Phi(\sigma) \right)=(\prod^3_{v=1}c_v i_v) \Phi(\sigma).
\end{align*}
By assumption, we know there are $c_v$ disjoint cycles with length $i_v$ in $\sigma_v$, it means the order of $p_{i_v}$ in the monomial $\Phi(\sigma_v)$ is $c_v$. So, when we calculate $\frac{\partial \Phi(\sigma_v)}{\partial p_{i_v}}$, we will have a coefficient $c_v$, i.e.
\begin{align*}
p_{i_1+i_2+i_3} \frac{\partial \Phi(\sigma_1)}{\partial p_{i_1}}\frac{\partial \Phi(\sigma_2)}{\partial p_{i_2}}\frac{\partial \Phi(\sigma_3)}{\partial p_{i_3}}=(\prod^3_{v=1}c_v) \Phi(\sigma).
\end{align*}
So, we have
\begin{align*}
& i_1 i_2 i_3 p_{i_1+i_2+i_3} \frac{\partial \Phi(\sigma_1)}{\partial p_{i_1}}\frac{\partial \Phi(\sigma_2)}{\partial p_{i_2}}\frac{\partial \Phi(\sigma_3)}{\partial p_{i_3}}\\
 &=\left( \sum_{(\delta_2,...,\delta_k,\sigma,\varepsilon) \in \mathcal{OA}^3_1(\alpha',\delta_2,...,\delta_k,\sigma',\varepsilon)}\Phi(\sigma) \right).
\end{align*}
Finally, we sum over all elements in $\mathcal{OB}^3_1(\alpha')$ and we get the following formula
\begin{align*}
& \sum_{(\delta_2,...,\delta_k,\sigma',\varepsilon) \in \mathcal{OB}^3_1(\alpha')}i_1 i_2 i_3 p_{i_1+i_2+i_3} \frac{\partial \Phi(\sigma_1)}{\partial p_{i_1}}\frac{\partial \Phi(\sigma_2)}{\partial p_{i_2}}\frac{\partial \Phi(\sigma_3)}{\partial p_{i_3}}\\
& =\sum_{\mathcal{OA}^3_1(\alpha',\delta_2,...,\delta_k,\sigma',\varepsilon)} \sum_{(\delta_2,...,\delta_k,\sigma,\varepsilon) \in \atop \mathcal{OA}^3_1(\alpha',\delta_2,...,\delta_k,\sigma',\varepsilon)} \Phi(\sigma)\\
& =\left( \sum_{(\delta_2,...,\delta_k,\sigma,\varepsilon) \in \mathcal{OA}^3_1(\alpha',i_1,i_2,i_3)}\Phi(\sigma) \right).
\end{align*}
\end{proof}

\begin{lemma}\label{5009}
We have the following equations
\begin{align*}
    & (\frac{\partial \widetilde{F}_3}{\partial u})_1=\frac{1}{3}\sum_{i,j,k \geq 1}( ijkp_{i+j+k}\frac{\partial \widetilde{F}_3}{\partial p_{i}}\frac{\partial \widetilde{F}_3}{\partial p_{j}}\frac{\partial \widetilde{F}_3}{\partial p_{k}}),\\
    & (\frac{\partial \widetilde{F}_3}{\partial u})_2=\sum_{i,j,k \geq 1}(i(j+k)p_{i+k} p_{j} \frac{\partial \widetilde{F}_3}{\partial p_{i}}\frac{\partial \widetilde{F}_3}{\partial p_{j+k}}),\\
    & (\frac{\partial \widetilde{F}_3}{\partial u})_3=\frac{1}{3}\sum_{i,j,k \geq 1}((i+j+k)p_i p_j p_k \frac{\partial \widetilde{F}_3}{\partial p_{i+j+k}}).
\end{align*}
\end{lemma}

\begin{proof}
We only give the proof of the first equation. The proofs of the other equations are similar.

First, we want to find the relation between the generating function of sets $\mathcal{OB}_1^3(\alpha')$ and the generating functions of sets $\mathcal{OA}_1^3(\alpha)$. Recall that $(\tilde{F}_3)^3$ counts the elements in $\mathcal{OB}_1^3(\alpha')$ by Lemma \ref{5007}. With the same notations in Construction \ref{50008}, let $n_v$ be the number of elements in $X_v$. Recall we have $\mu^3(\alpha')-2=\mu^3(\alpha)-1$ in the proof of Lemma \ref{5007}. By Lemma \ref{50010}, we get the following formula
\begin{align*}
&\left( \sum_{(\delta_2,...,\delta_k,\sigma',\varepsilon) \in \mathcal{OB}^3_1(\alpha')} i_1 i_2 i_3 p_{i_1+i_2+i_3} \frac{\partial \Phi(\sigma_1)}{\partial p_{i_1}}\frac{\partial \Phi(\sigma_2)}{\partial p_{i_2}}\frac{\partial \Phi(\sigma_3)}{\partial p_{i_3}} \right) \frac{z^n}{n!}\frac{u^{\mu^3(\alpha')-2}}{(\mu^3(\alpha')-2)!}\\
=&\left( \sum_{(\delta_2,...,\delta_k,\sigma,\varepsilon) \in \mathcal{OA}^3_1(\alpha',\alpha,i_1,i_2,i_3)}\Phi(\sigma) \right) \frac{z^n}{n!}\frac{u^{\mu^3(\alpha)-1}}{(\mu^3(\alpha)-1)!}.
\end{align*}
We sum over all sets $\mathcal{OB}^3_1(\alpha')$ and positive integers $i_1,i_2,i_3$ on the left side. Also, we sum over all all sets $\mathcal{OA}^3_1(\alpha',\alpha,i_1,i_2,i_3)$ and positive integers $i_1,i_2,i_3$ on the right side. We have the following equation
\begin{align*}
\sum_{i,j,k \geq 1}( ijkp_{i+j+k}\frac{\partial \widetilde{F}_3}{\partial p_{i}}\frac{\partial \widetilde{F}_3}{\partial p_{j}}\frac{\partial \widetilde{F}_3}{\partial p_{k}}) =\sum_{n \geq 1}\sum_{\alpha\vdash n}|\mathcal{OA}_1^3(\alpha)|\frac{z^n}{n!}\frac{u^{\mu^3(\alpha)-1}}{(\mu^3(\alpha)-1)!}p_\alpha.
\end{align*}
Indeed, the proof of Lemma \ref{5007} gives the left side of the equation. Then, by Lemma \ref{5010}, we know
\begin{align*}
6|\tilde{\mathcal{A}}^3_1(\alpha)|=|\widetilde{\mathcal{WOA}}^3_1(\alpha)|=2|\mathcal{OA}^3_1(\alpha)|.
\end{align*}
Recall that $(\frac{\partial \widetilde{F}_3}{\partial u})_1$ counts the elements in $\tilde{\mathcal{A}}^3_1(\alpha)$ by Lemma \ref{5005}. Hence, we have
\begin{align*}
6(\frac{\partial \widetilde{F}_3}{\partial u})_1=2\sum_{i,j,k \geq 1}( ijkp_{i+j+k}\frac{\partial \widetilde{F}_3}{\partial p_{i}}\frac{\partial \widetilde{F}_3}{\partial p_{j}}\frac{\partial \widetilde{F}_3}{\partial p_{k}}),
\end{align*}
which gives the first equation.
\end{proof}

\begin{theorem}\label{503}
The generating function $F_3$ satisfies the following relation
\begin{align*}
& \frac{1}{3}\sum_{i,j,k \geq 1}( (i+j+k)p_i p_j p_k \frac{\partial F_3}{\partial p_{i+j+k}}\\
& +ijkp_{i+j+k} \frac{\partial F_3}{\partial p_i}\frac{\partial F_3}{\partial p_j}\frac{\partial F_3}{\partial p_k}+3(i+j)kp_i p_{j+k} \frac{\partial F_3}{\partial p_{i+j}}\frac{\partial F_3}{\partial p_k})\\
& -\frac{1}{2}(z \frac{\partial F_3}{\partial z}+\sum_{i \geq 1} p_i \frac{\partial F_3}{\partial p_i}-2F_3)=0.
\end{align*}
\end{theorem}

\begin{proof}
Recall the generating function $\widetilde{F}_3(z,u,p_1,p_2,...)$ in Definition \ref{5004},
\begin{align*}
\widetilde{F}_3(z,u,p_1,p_2,...)=\sum_{n \geq 1}\sum_{\alpha\vdash n}h^{3}(\alpha)\frac{z^n}{n!}\frac{u^{\mu^3(\alpha)}}{\mu^3(\alpha)!}p_\alpha.
\end{align*}
Consider the following generating function
\begin{align*}
\frac{\partial \widetilde{F}_3}{\partial u}=\sum_{n \geq 1}\sum_{\alpha\vdash n}h^{3}(\alpha)\frac{z^n}{n!}\frac{u^{\mu^3(\alpha)-1}}{(\mu^3(\alpha)-1)!}p_\alpha.
\end{align*}
By Lemma \ref{5009}, we have the following equation
\begin{align*}
\frac{\partial \widetilde{F}_3}{\partial u}=\frac{1}{3}&\sum_{i,j,k \geq 1}((i+j+k)p_i p_j p_k \frac{\partial \widetilde{F}_3}{\partial p_{i+j+k}}+3(i+j)kp_i p_{j+k} \frac{\partial \widetilde{F}_3}{\partial p_{i+j}}\frac{\partial \widetilde{F}_3}{\partial p_{k}}\\
&+ijkp_{i+j+k} \frac{\partial \widetilde{F}_3}{\partial p_i}\frac{\partial \widetilde{F}_3}{\partial p_j}\frac{\partial \widetilde{F}_3}{\partial p_k}).
\end{align*}
By Definition \ref{5004}, we know $\widetilde{F}_3|_{u=1}=F_3$. So, let $u=1$. Then, the RHS is
\begin{align*}
\text{Right Side}=&\frac{1}{3}\sum_{i,j,k \geq 1}((i+j+k)p_i p_j p_k \frac{\partial F_3}{\partial p_{i+j+k}}+3(i+j)kp_i p_{j+k} \frac{\partial F_3}{\partial p_{i+j}}\frac{\partial F_3}{\partial p_{k}}\\
&+ijkp_{i+j+k} \frac{\partial F_3}{\partial p_i}\frac{\partial F_3}{\partial p_j}\frac{\partial F_3}{\partial p_k}).
\end{align*}
By simple calculations, we have
\begin{align*}
& \frac{\partial \widetilde{F}_3}{\partial u}=\sum_{n \geq 1}\sum_{\alpha\vdash n}\mu^3(\alpha)h^{3}(\alpha)\frac{z^n}{n!}\frac{u^{\mu^3(\alpha)-1}}{\mu^3(\alpha)!}p_\alpha,\\
& z\frac{\partial F_3}{\partial z}=\sum_{n \geq 1}\sum_{\alpha\vdash n}n h^{3}(\alpha)\frac{z^n}{n!}\frac{1}{\mu^3(\alpha)!}p_\alpha,\\
& \sum_{i \geq 1}p_i \frac{\partial F_3}{\partial p_i}=\sum_{n \geq 1}\sum_{\alpha\vdash n}l(\alpha)h^{3}(\alpha)\frac{z^n}{n!}\frac{1}{\mu^3(\alpha)!}p_\alpha.
\end{align*}
where $l(\alpha)$ is the length for the partition $\alpha$. By Lemma \ref{408} and Remark \ref{5000001}, we know
\begin{align*}
\mu^3(\alpha)=\frac{n+l(\alpha)-2}{2}.
\end{align*}
Hence, we have the following formula
\begin{align*}
\frac{\partial \widetilde{F}_3}{\partial u}\mid_{u=1}=\frac{1}{2}(z \frac{\partial F_3}{\partial z}+\sum_{i \geq 1} p_i \frac{\partial F_3}{\partial p_i}-2F_3).
\end{align*}
So, we have
\begin{align*}
\text{Left Side}=\frac{1}{2}(z \frac{\partial F_3}{\partial z}+\sum_{i \geq 1} p_i \frac{\partial F_3}{\partial p_i}-2F_3).
\end{align*}
Combining LHS and RHS, we obtain the theorem.
\end{proof}

Now let's go back to the $W$-operator. The $W$-operators $W([n])$ are well-defined differential operators on polynomial ring $\mathbb{C}[p_1,p_2,...]$. Some terms in $W([n])$ contain higher derivatives. For example, in $W([2])$, we have a summation
\begin{align*}
\frac{1}{2}\sum_{i \geq 1}\sum_{j\geq 1}ijp_{i+j}\frac{\partial^2}{\partial p_i \partial p_j},
\end{align*}
which contains second derivatives. If we change the higher derivatives into the product of first derivatives, we will get a new operator. We still take $W([2])$ as an example,
\begin{align*}
\widetilde{W}([2])= \frac{1}{2}\sum_{i \geq 1}\sum_{j\geq 1}(ijp_{i+j}\frac{\partial}{\partial p_i} \times \frac{\partial}{\partial p_i}+(i+j)p_i p_j \frac{\partial}{\partial p_{i+j}}).
\end{align*}
As an operator, it means
\begin{align*}
\widetilde{W}([2])F= \frac{1}{2}\sum_{i \geq 1}\sum_{j\geq 1}(ijp_{i+j}\frac{\partial F}{\partial p_i} \frac{\partial F}{\partial p_i}+(i+j)p_ip_j \frac{\partial F}{\partial p_{i+j}}),
\end{align*}
where $F \in \mathbb{C}[p_1,p_2,...]$.
\begin{definition}\label{504}
Define $\widetilde{W}([d])$ by replacing all higher derivatives in $W([d])$ by the products of first derivative operators as mentioned above.
\end{definition}

\begin{remark}
With the new notation $\widetilde{W}([2])$, we can rewrite the formula in Lemma \ref{203} as the following equation,
\begin{align*}
\widetilde{W}([2])F_2 - z \frac{\partial F_2}{\partial z} - \sum_{i \geq 1}p_i \frac{\partial F_2}{\partial p_i}+2F_2=0.
\end{align*}
Finally, we rewrite Theorem \ref{503} as following:
\end{remark}
\begin{theorem}\label{505}
The generating function $F_3(z,p)$ satisfies the following relation
\begin{align*}
& \widetilde{W}([3])F_3-\sum_{i,j,k \geq 1}(i+j+k)p_{i+j+k} \frac{\partial F_3}{\partial p_{i+j+k}}\\
=& \frac{1}{2}(z \frac{\partial F_3}{\partial z}+\sum_{i \geq 1} p_i \frac{\partial F_3}{\partial p_i}-2F_3).
\end{align*}
\end{theorem}

\section{Conjecture for General Case}
By Theorem \ref{306}, we know that $W([d])$ can be written as the sum of $d!$ summations, each of which is uniquely determined by a permutation in $S_d$ (see Theorem 3.15 in \cite{Sun1}). For any permutation $\beta \in S_d$, denote by $FS_\beta$ the summation corresponding to $\beta$. We have
\begin{align*}
W([d])=\sum_{\beta \in S_d} FS_\beta.
\end{align*}
Now we give some examples about $FS_\beta$. The first example is $W([2])$. From Example \ref{3006}, we know
\begin{align*}
W([2])= \frac{1}{2}\sum_{i \geq 1}\sum_{j\geq 1}(&ijp_{i+j}\frac{\partial^2}{\partial p_i \partial p_j} & (1 2)\\
&+(i+j)p_i p_j \frac{\partial}{\partial p_{i+j}}) & (1).
\end{align*}
The first summation
\begin{align*}
FS_{(12)}=\frac{1}{2}\sum_{i \geq 1}\sum_{j\geq 1}(ijp_{i+j}\frac{\partial^2}{\partial p_i \partial p_j})
\end{align*}
corresponds to the permutation $(12)$ and the second summation
\begin{align*}
FS_{(1)(2)}=\frac{1}{2}\sum_{i \geq 1}\sum_{j\geq 1}((i+j)p_i p_j \frac{\partial}{\partial p_{i+j}})
\end{align*}
corresponds to the permutation $(1)(2)$.

The second example is $W([3])$ (see Example \ref{3006}). We have
\begin{align*}
		W([3]) = \frac{1}{3}\sum_{i_1,i_2,i_3 \geq 1}
		(&i_1i_2i_3 p_{i_1+i_2+i_3}\frac{\partial^3}{\partial p_{i_1} \partial p_{i_2}\partial p_{i_3}}+ &(321) \\
		+&i_1(i_2+i_3)p_{i_1+i_3}p_{i_2}\frac{\partial^2}{\partial p_{i_1} \partial p_{i_2+i_3}}+ &(13)(2)\\
		+&i_2(i_1+i_3)p_{i_1+i_2}p_{i_3}\frac{\partial^2}{\partial p_{i_2} \partial p_{i_1+i_3}}+ &(12)(3)\\
		+&i_3(i_1+i_2)p_{i_3+i_2}p_{i_1}\frac{\partial^2}{\partial p_{i_3} \partial p_{i_1+i_2}}+ &(1)(23)\\
		+&(i_1+i_2+i_3)p_{i_1}p_{i_2}p_{i_3} \frac{\partial}{\partial p_{i_1+i_2+i_3}}+ &(3)(2)(1)\\
		+&(i_1+i_2+i_3)p_{i_1+i_2+i_3}\frac{\partial}{\partial p_{i_1+i_2+i_3}}) &(123).
\end{align*}
Each summation $FS_\beta$ is the sum of terms. All terms in $FS_\beta$ have the same polynomial degree and the same order of the differential part. for each summation $FS_\beta$, we define its degree. Hence, we can define the degree $FS_\beta$ as following.
\begin{definition}\label{601}
Given any summation $FS_\beta$ of $W([d])$, $dP(FS_\beta)$ is the degree of its polynomial part and $dD(FS_\beta)$ is the order of its derivative part. The degree of the summation $FS_\beta$ is $d(FS_\beta)=dP(FS_\beta)+dD(FS_\beta)$.
\end{definition}
Let's consider $W([3])$. Five of the six summations have degree $4$ and the summation $FS_{(123)}$ is of degree $2$. If we go back to $W([2])$, all summations are of degree $3$. In fact, the degree of the summations in $W([d])$ is at most $d+1$. We discuss the degree of $W$-operator in \cite{Sun4}. Since $W([d])$ is the sum of $d!$ summations, so is $\widetilde{W}([d])$. We denote by $\widetilde{FS}_\beta$ the corresponding summation in $\widetilde{W}([d])$. We define another operator $\widetilde{HW}([d])$ as following.
\begin{definition}\label{602}
\begin{align*}
\widetilde{HW}([d])=\sum_{\beta \in S_d \atop FS_\beta \text{ is of degree }d+1} \widetilde{FS}_\beta.
\end{align*}
\end{definition}
Now we can state our conjecture.
\begin{conjecture}\label{603}
\begin{equation}\label{conj1}
\frac{\partial \widetilde{F}_d}{\partial u}=\widetilde{HW}([d])\widetilde{F}_d,
\end{equation}
where $\widetilde{F}_d=\sum_{n \geq 1}\sum_{\alpha\vdash n}h^d(\alpha)\frac{z^n}{n!}\frac{u^{\mu^d(\alpha)}}{\mu^d(\alpha)!}p_\alpha$ (see Definition \ref{401}, \ref{402}, \ref{5004}).
\end{conjecture}
If we take $u=1$, RHS of \eqref{conj1} is $\widetilde{HW}([d])F_d$, where
\begin{align*}
F_d=\sum_{n \geq 1}\sum_{\alpha\vdash n}h^d(\alpha)\frac{z^n}{n!}\frac{1}{\mu^d(\alpha)!}p_\alpha.
\end{align*}
To determine the LHS of \eqref{conj1}, we need the following lemma proved by Goulden and Jackson \cite{MR1797682}.
\begin{lemma}\label{604}
Let $\alpha$ be a partition of $n$. We have
\begin{align*}
\mu^{d}(\alpha)=\frac{n+l(\alpha)-2}{d-1},
\end{align*}
where $l(\alpha)$ is the length of the partition.
\end{lemma}
Hence, the LHS of \eqref{conj1} is
\begin{equation}\label{conj2}
\frac{\partial \widetilde{F}_d}{\partial u}|_{u=1}=\frac{1}{d-1}\left( z\frac{\partial F_d}{\partial z}+ \sum_{i \geq 1}p_i \frac{\partial F_d}{\partial p_i} -2F_d \right).
\end{equation}
We want to make a remark about Lemma \ref{604}. If $\mu^{d}(\alpha)$ exists, $\mu^{d}(\alpha)$ is a positive integer. Sometimes, $\mu^d(\alpha)$ does not exist. But, we can extend the definition of $\mu^{d}(\alpha)$ to any partition $\alpha$ in the same way as Remark \ref{5000001}.

If Conjecture \ref{603} is true, we have the following corollary.
\begin{corollary}\label{605}
\begin{align*}
\widetilde{HW}([d])F_d=\frac{1}{d-1}\left( z\frac{\partial F_d}{\partial z}+ \sum_{i \geq 1}p_i \frac{\partial F_d}{\partial p_i} -2F_d \right).
\end{align*}
\end{corollary}

\end{document}